\documentclass{ws-ijbc}
\usepackage{ws-rotating}     
\usepackage{graphicx}
\usepackage{epstopdf}
\usepackage{subfigure}

\renewcommand{\labelitemi}{$\bullet$} 

\newcommand{\bydef}{\,\stackrel{\mbox{\tiny\textnormal{\raisebox{0ex}[0ex][0ex]{def}}}}{=}\,}
\newcommand{\R}{\mathbb{R}}
\newcommand{\Z}{\mathbb{Z}}
\newcommand{\cC}{\mathcal{C}_{\alpha,k\pm1}^{(i)}}

\begin{document}
\title{Chebyshev-Taylor parameterization of stable/unstable manifolds for 
periodic orbits: implementation and applications 
}
\author{J.D. Mireles James}
\author{Maxime Murray}
\address{Department of mathematical sciences, Florida Atlantic University, 777 glades road\\
Boca Raton, Florida, 33431, United States of America\\
jmirelesjames@fau.edu \\
mmurray2016@fau.edu
}

\maketitle

\begin{abstract}
This paper develops Chebyshev-Taylor spectral methods for studying 
stable/unstable manifolds attached to periodic solutions of differential equations.
The work exploits the parameterization method
-- a general functional analytic
framework for studying invariant manifolds. Useful features of the parameterization 
method include the fact that it 
can follow folds in the embedding, recovers the dynamics on the 
manifold through a simple conjugacy, and admits a natural notion of 
a-posteriori error analysis.  Our approach begins by 
deriving a recursive system of linear differential equations 
describing the Taylor coefficients of the invariant manifold.
We represent periodic solutions of these equations as
solutions of coupled systems of boundary value problems. 
We discuss the implementation and performance of the method
for the Lorenz system, and for the planar
circular restricted three and four body problems.
We also illustrate the use of the method as a tool for computing 
cycle-to-cycle connecting orbits.  
\end{abstract}

\begin{center}
\end{center}

\section{Introduction}
Periodic solutions of differential equations are the basic building blocks of 
recurrence in nonlinear dynamics.  Moreover, hyperbolic periodic orbits and 
their heteroclinic/homoclinic connections are natural generators of chaotic motions.
Since a heteroclinic/homoclinic orbit 
will approach a periodic solution along its local stable/unstable manifolds,  
computational methods for studying these manifolds are of great interest.  
A schematic description of the stable manifold of a periodic orbit, 
beside an actual stable manifold in the Lorenz system are illustrated 
in Figure \ref{fig:stableManifold}.  See any of the works 
\cite{MR2163533, MR2475705, MR1700577, MR1705705, MR1630282, 
MR2136745, MR3026237, MR1981055,MR1702117, MR1054714, simoNotes}
for more discussion, but we caution
that any such list will hardly scratch the surface of the relevant literature.

The stable/unstable normal bundles of a periodic solution 
approximate the stable/unstable manifolds to first order, and 
these bundles are obtained by studying the 
equations of first variation -- or equivalently  -- by solving 
certain periodic eigenvalue problems.  Higher order jets
of the invariant manifold could be studied  
via higher order equations of variation, however the complexity of 
these equations grows exponentially with the order of the jet.  
More efficient methods for studying the jets are obtained by 
reformulating the invariant manifold as the solution of an operator equation, and 
studying the operator equation via numerical methods.

The parameterization method is a general functional analytic 
framework for studying invariant manifolds
\cite{MR1976079, MR1976080, MR2177465, MR3309008, MR2967458, 
MR3082311, MR2240743, MR2289544, MR2299977} whose goal 
is to find a chart/covering map conjugating the dynamics on the invariant 
manifold to a simple and well understood model (correctly choosing this 
model is part of the method).   By viewing the conjugacy as an operator 
equation for the unknown chart, the problem is susceptible to numerical methods.  
We will see below that the operator equation of interest in the present work is actually a 
PDE with prescribed periodic data.  

This operator equation is referred to as the \textit{invariance equation}.
The unknown parameterization is not required to be the graph of a function,
and hence is able to follow folds in the embedding.  See again Figure \ref{fig:stableManifold}.
Since the invariance equation is based on a dynamical conjugacy, the 
parameterization method recovers the dynamics on the manifold
in addition to the embedding.   By now there is a small industry 
devoted to the parameterization method, and a review of the literature 
would take us far afield.  We refer the interested reader to the 
recent book on the subject \cite{MR3467671}, where many examples 
and much more complete discussion of the literature is found.

\begin{figure}
    \centering
    \subfigure{{\includegraphics[width=.45\textwidth]{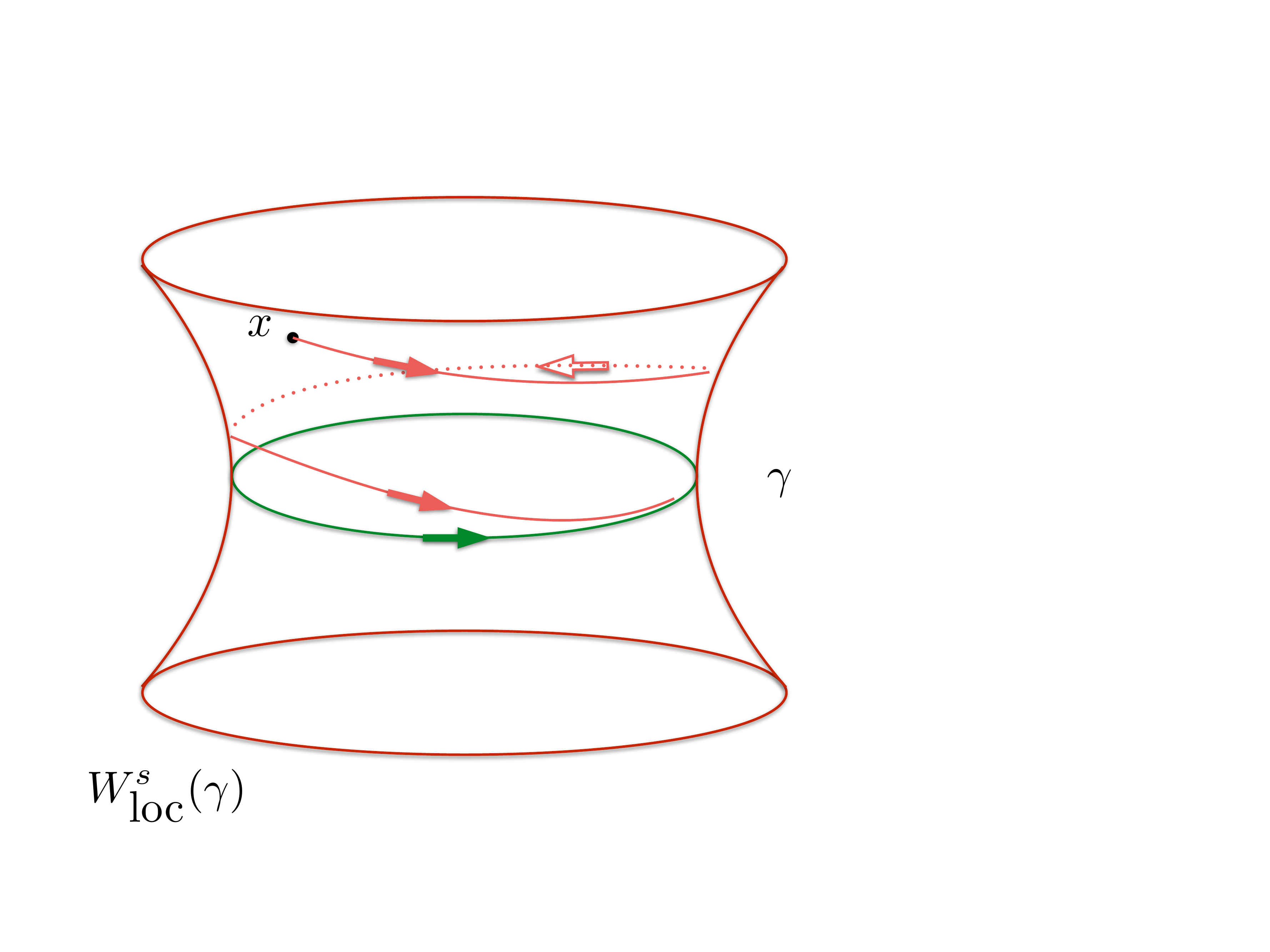} }}
     \subfigure{{\includegraphics[width=.45\textwidth]{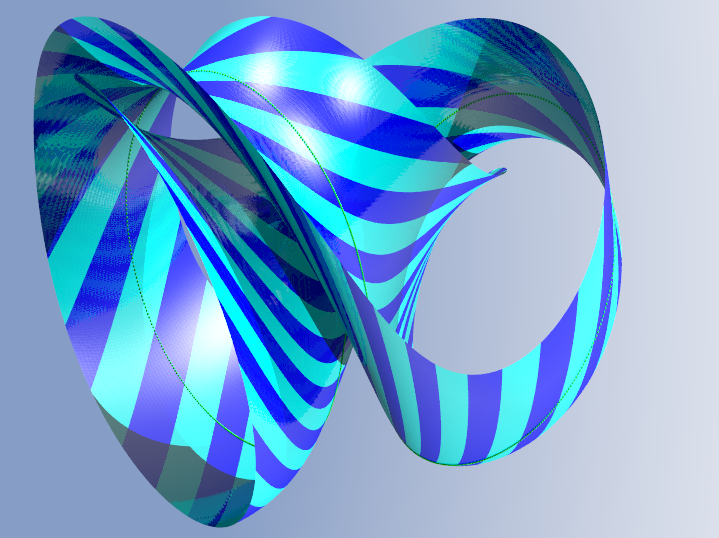} }}
\caption{The local stable manifold attached to a periodic orbit $\gamma(t)$: (LEFT)
schematic representation of the stable manifold, i.e. it is the union of all solutions
accumulating at $\gamma$ as $t \to \infty$ (periodic orbit 
shown here as the green loop).
 The unstable manifold is the same but with 
orbits approaching the periodic orbit asymptotically in backward time.
(RIGHT) local stable manifold of an actual periodic orbit near the 
Lorenz attractor.  This manifold is computed using the techniques of the 
present work.  In particular, changes in color in the figure indicate the 
transitions from one Chebyshev domain to another.
}\label{fig:stableManifold}
\end{figure}

The present work is not the first numerical treatment of parameterized
stable/unstable manifolds attached to periodic orbits, and indeed we  build on the earlier studies of  
\cite{MR2177465, MR3304254, MR2551254, MR3118249}.
 As in these earlier studies, 
we make a formal Taylor series arguments which analytically reduces the invariance equation 
before making any numerical computations.  
Since the coefficients of the Taylor series are themselves periodic functions, the formal
power matching scheme leads to so called \textit{homological equations}
describing the unknown coefficients.  In the present case of a periodic orbit,   
the homological equations are linear ordinary differential equations
with periodic coefficients and periodic forcing.  We solve these linear homological 
equations recursively using numerical spectral methods. 
By computing the formal series solution to high order we obtain 
an approximate solution which describes the 
stable/unstable manifold far from the periodic orbit.  

In the earlier works just cited the differential equations -- describing 
the periodic orbit, describing the normal bundles, and the homological equations
describing the higher order jets -- are all solved using Fourier spectral methods.  
Fourier methods are both efficient and accurate when applied to periodic solutions
of moderate length.  This efficiency is due in part to the fact that differentiation is a 
diagonal operation in the transform domain, and in part to the fact 
that the FFT speeds up evaluation of nonlinearities.
However, the decay rate of the Fourier coefficients
gets increasingly slow as the period/harmonic complexity
of the orbit grow.  In practice this means that it is necessary to compute
more and more Fourier coefficients, and  for long enough orbits the 
Fourier approximation becomes 
impractical. 

There is much recent interest in numerical methods
based on Chebyshev spectral approximation of  solutions to
boundary value problems.  We refer the interested reader to 
 \cite{MR2465699, MR2767023, MR2384813, MR3012510}
 and the references discussed there in.  The present work 
 builds on techniques developed by a number of authors which use 
 Chebyshev spectral methods 
to compute long periodic solutions of differential equations
\cite{MR3392421, LessardReinhardt, RayJB}.
The merit of this approach is that the Chebyshev spectral methods
posess many of the advantages of Fourier series 
-- for example differentiation is a tri-diagonal operation in the transform 
domain and the fast cosine transform is available for evaluating nonlinearities --
but Chebyshev series apply to non-periodic boundary value problems.  
Treating a periodic solution as
a series of coupled boundary value problems -- on smaller domains --
facilitates control of the decay rates of the coefficients.

Motivated by these developments,  the present work applies 
Chebyshev methods not only the 
periodic orbit -- but also to the computation of 
the normal bundles and the homological equations
for the higher order jets. The result is a 
computational method for finding Chebyshev-Taylor 
expansions of the local stable/unstable manifolds attached to
periodic orbits.  Our method applies to more complicated orbits and 
their attached invariant manifolds than could be studied using only
the Fourier-Taylor approach.

\begin{remark}[Connecting orbits and extensions of local stable/unstable manifolds]
{\em 
Of course computing local stable/unstable manifolds attached to periodic orbits
is only a means to an end.  In applications we are often interested in either using the 
local manifolds to compute connecting orbits, or to grow larger 
portions of the invariant manifold in order to study the global dynamics.  
While in the present work we do consider a number of 
example computations for connecting orbits, we do not make any serious 
attempt to numerically grow larger local manifolds. 
This is because the literature on computational methods for 
growing invariant manifolds is extensive and well developed.   
The interested reader will want to consult the 
review paper \cite{MR2136745} for a thorough overview of the literature, and 
will find other powerful methods and fuller discussion in 
\cite{MR2179495, MR1870261, MR1981055}.  We only note that the methods developed
in the present work could be combined with existing continuation methods 
for even better results.  This is especially true for methods 
 which exploit the curvature or other differential geometric properties
of the manifold.  
}
\end{remark}

\begin{remark}[Automatic differentiation and polynomial nonlinearities] \label{rem:autoDiff}
{\em
Multiplication of Taylor and Chebyshev series is
straight forward thanks to the Cauchy product in the former case and the 
discrete cosine convolution operation in the later.
Then formal series manipulations for polynomial nonlinearities
are especially transparent in Chebyshev-Taylor bases.   
In the present work we are interested 
in applications coming from celestial mechanics which involve non-polynomial 
vector fields.  In order to simplify matters we
exploit methods of \textit{automatic differentiation} and transform 
to the polynomial setting, albeit in a higher dimensional phase space. 
This is discussed in detail in Section \ref{sec:autoDiff}.

The use of 
automatic differentiation is a convenience rather than a necessity, as the FFT 
could be used to evaluate general nonlinearities. In fact, even after automatic differentiation 
we use the fast cosine transform to evaluate higher order polynomial nonlinearities
in the present work.  Nevertheless, the use of automatic differentiation
in the present work simplifies the implementation details of our algorithms --  as
all our computations are reduced to Newton's method for large polynomial systems.   
Automatic differentiation also simplifies a-posteriori error analysis for the method, 
which when followed to its logical conclusion provides mathematically 
rigorous validated error bounds for the parameterizations. 
}
\end{remark}

\section{Review of the parameterization method}
As already mentioned in the introduction, the parameterization method is 
much more general than what we actually use in the present work.
We refer the reader again to the book 
\cite{MR3467671}.  In the following section we review some basic notions
in the very simple setting of an orientable local manifold associated with
one stable/unstable Floquet  exponent.  Generalities such as 
multiple stable/unstable exponents, complex conjugate exponents, and 
non-orientable bundles are discussed in detail in 
\cite{MR3304254}.  The methods of the present work apply in these more general 
setting with only obvious modifications.  We focus on the one dimensional case 
to simplify the exposition.

\subsection{Parameterization of stable/unstable manifolds attached to periodic orbits}
Let $\Omega \subset \mathbb{R}^M$ be an open set and $g \colon \Omega \to \mathbb{R}^M$
be a real analytic vector field.  
Suppose that $\gamma \colon \mathbb{R} \to \mathbb{R}^M$
is a $T$-periodic solution of the first order ordinary differential equation
\[
\dot x = g(x),
\]
that is we assume that $\gamma'(t) = g(\gamma(t))$ with  
$\gamma(t + T) = \gamma(t)$  for all $t \in \mathbb{R}$.  
Suppose also that $\gamma$ has one stable Floquet exponent
\[
\lambda \in \mathbb{R}, \quad \quad \quad \mbox{with } \quad \quad \quad 
\lambda < 0,
\]
so that (by the stable manifold theorem)
there exists a two dimensional manifold of solutions which converge exponentially fast to 
the periodic orbit $\gamma$. 
Let $v \colon \mathbb{R} \to \mathbb{R}^M$ denote the stable normal bundle of $\gamma(t)$,
associated with the exponent $\lambda$. 
We assume that $v$ is an orientable bundle, so that $v(t)$ is $T$ periodic as well.
We note that $(v, \lambda)$ solve the eigenvalue problem 
\[
\frac{d}{dt} v(t) = Dg(\gamma(t)) v(t) - \lambda v(t),
\]
subject to some normalization, perhaps $\|v(t)\| = 1$ for all $t \in \mathbb{R}$
(though in numerical applications we will choose other normalizations).

The goal of the parameterization method is to find a smooth function 
$P \colon [0, T]  \times [-1, 1] \to \mathbb{R}^M$ solving the 
invariance equation
\begin{equation}\label{eq:InvarianceEquation}
 \frac{\partial}{\partial t}P(t,\sigma) + \lambda\sigma \frac{\partial}{\partial \sigma}P(t,\sigma) = g(P(t,\sigma)),
\end{equation} 
subject to the first order constraints
\begin{equation}
P(t, 0) = \gamma(t), 
\end{equation}
and
\begin{equation}\label{eq:P_returns_Orbit}
 \frac{\partial}{\partial \sigma} P(t,0) = v(t).
\end{equation}
Then geometric content of Equation \eqref{eq:InvarianceEquation}
is illustrated in Figure \ref{fig:parmMethodSketch}, but one easily checks 
that the image of $P$ is a stable manifold.

\begin{figure}
    \subfigure{{\includegraphics[width=.4\textwidth]{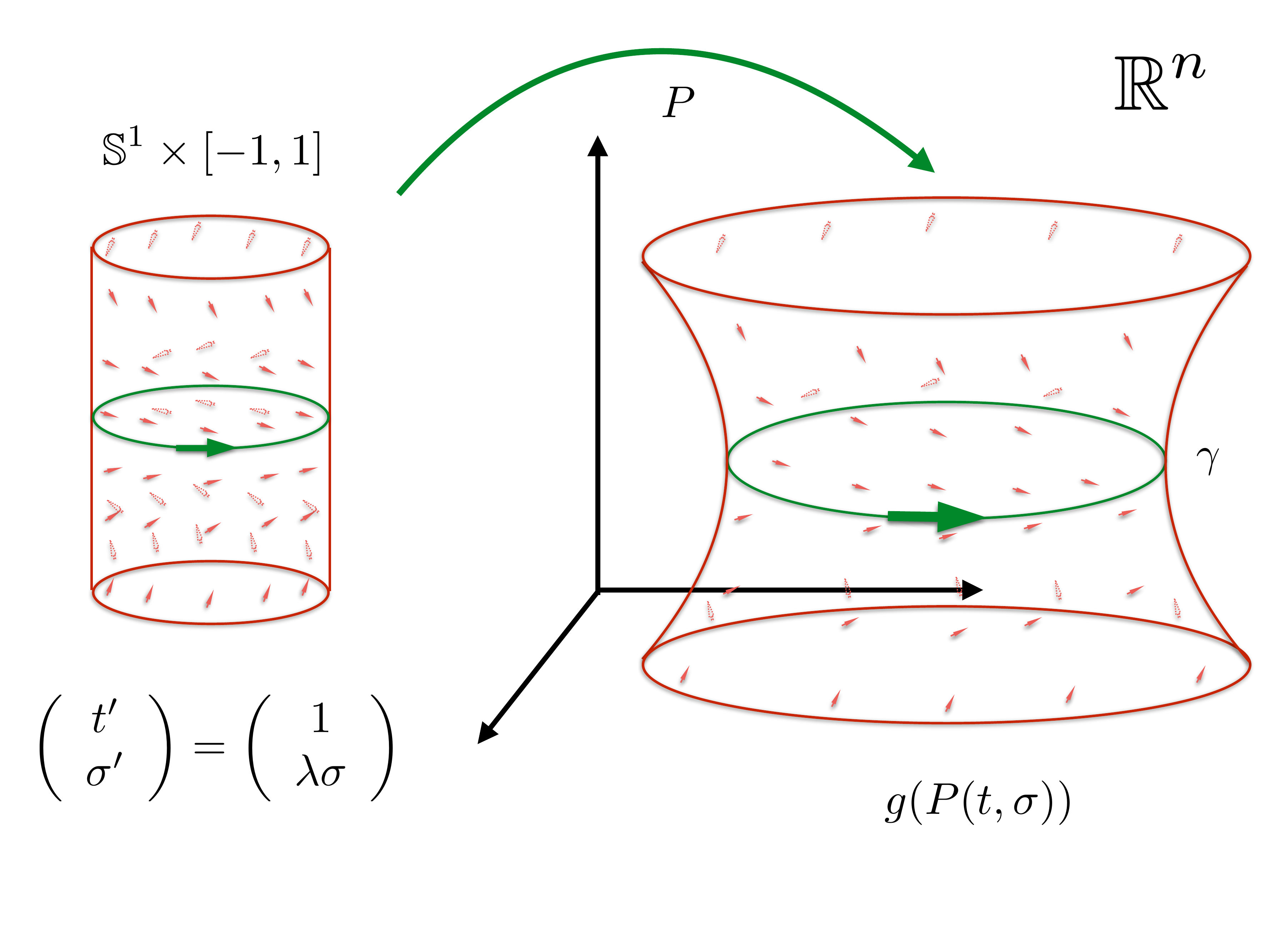}}}
     \subfigure{{\includegraphics[width=.4\textwidth]{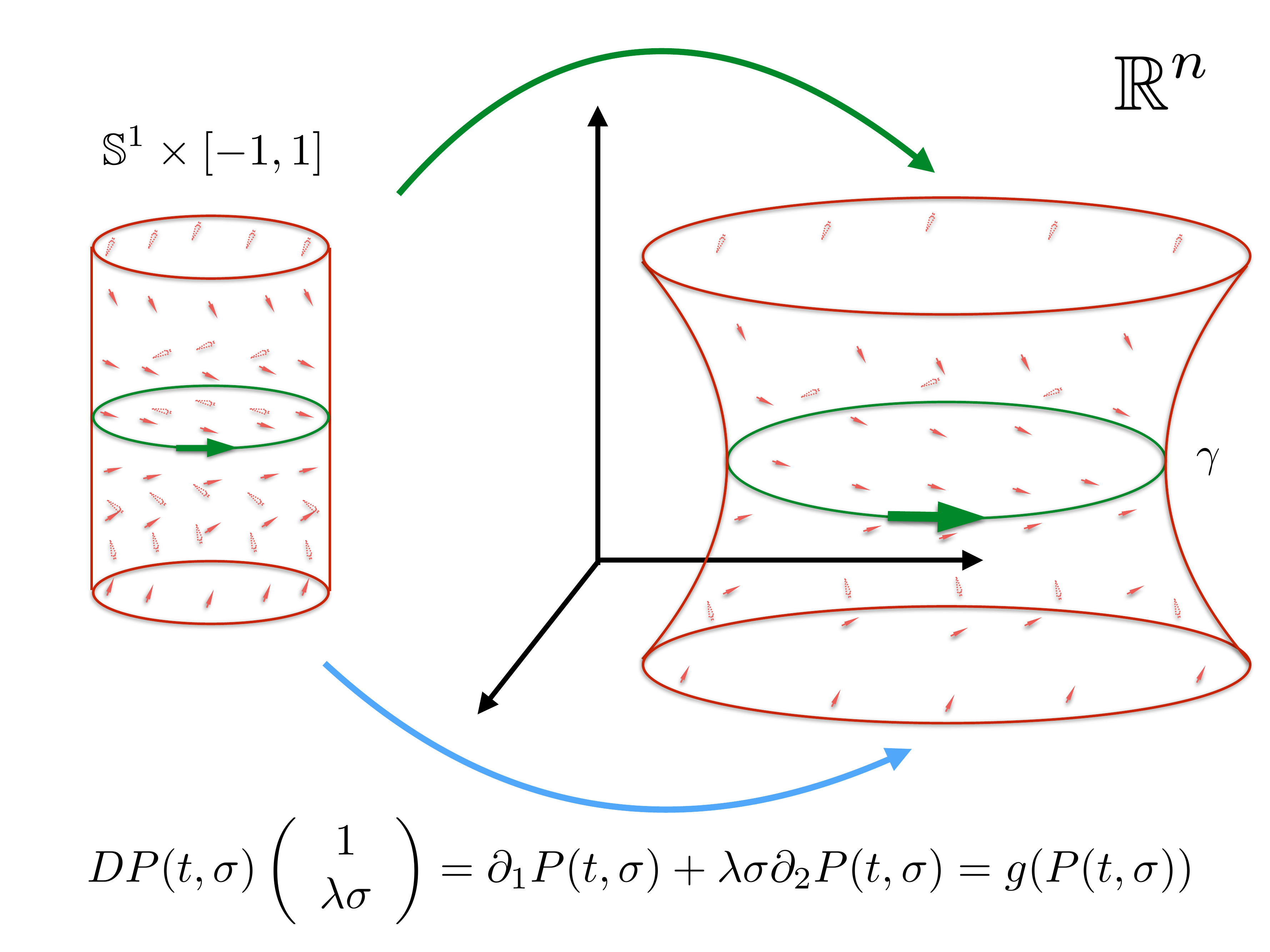} }}
\caption{Geometric meaning of Equation \eqref{eq:InvarianceEquation}: (LEFT) the 
model space for the stable manifold is the cylinder $\mathbb{S}^1 \times [-1, 1]$,
which we endow with the model vector field $t' = 1$, $\sigma' = \lambda \sigma$.
Here $\lambda$ is the stable Floquet exponent of the periodic orbit $\gamma(t)$. 
Now if $P$ is \textit{any} mapping 
from the cylinder into $\mathbb{R}^M$, then $g$ induces a vector field on the 
image of $P$ by restriction. (RIGHT) The derivative of $P$ pushes forward vector fields
defined on the cylinder.
The goal of the parameterization method is to find a mapping $P$ so that the 
push forward of the model dynamics is equal on the image of $P$
to the vector field induced by $g$.  If these vector fields are equal then they
have the same dynamics.  But the dynamics on the cylinder are completely understood.
See Figure \ref{fig:conjSketch}.}\label{fig:parmMethodSketch}
\end{figure}

Indeed, let $P$ be a smooth solution 
of  Equation \eqref{eq:InvarianceEquation} subject to the first order constraints.  
Choose any $\sigma_0 \in (-1, 1)$ and define the curve $x \colon [0, \infty) \to \mathbb{R}^M$
by 
\[
x(t) = P(t, e^{\lambda t} \sigma_0).
\]
Then 
\begin{align*}
\frac{d}{dt} x(t) &= DP(t, e^{\lambda t} \sigma_0) \left(
\begin{array}{c}
1 \\
e^{\lambda t} \lambda \sigma_0
\end{array}
\right) \\
&= \frac{\partial}{\partial t} P(t, e^{\lambda t} \sigma_0) + \lambda \sigma_0 e^{\lambda t}
\frac{\partial}{\partial \sigma} P(t, e^{\lambda t} \sigma_0) \\
&= g(P(t, e^{\lambda t} \sigma_0)) \\
& = g(x(t)),
\end{align*}
as $e^{\lambda t} \sigma_0 \in (-1, 1)$ for all $t \geq 0$.  Then $x(t)$ is a solution curve for
the differential equation having $x(0) = P(0, \sigma_0)$.  
Moreover for any $t_0 \in [0, L]$, since $P$ is continuous we have that 
\begin{align*}
\lim_{t \to \infty} P(t_0, e^{\lambda t} \sigma_0) 
&= P\left(t_0,  \lim_{t \to \infty} e^{\lambda t} \sigma_0 \right) \\
& = P(t_0, 0) \\
&= \gamma(t_0),
\end{align*}
that is, a point on the image of $P$ accumulates at the periodic orbit $\gamma$
with asymptotic phase $\gamma(t_0)$.  In particular, the image of $P$ is 
a local stable manifold for $\gamma$.

One can actually prove more.  For example 
if $P$ solves Equation \eqref{eq:InvarianceEquation} subject to the 
first order constraints, then $P$ actually satisfies the flow conjugacy
\begin{align}\label{eq:FlowConjugacy}
P(s+t,e^{\lambda t}\sigma)= \Phi(P(s,\sigma),t)
\end{align}
for all $t \geq 0$.  Here $\Phi$ is the flow generated by the vector field
$g$.  The meaning of this flow conjugacy is illustrated in Figure \ref{fig:conjSketch}.
The proof of the flow conjugacy is given for example in 
\cite{MR3304254}.
In the same reference it is shown that solutions of Equation \eqref{eq:InvarianceEquation}
are unique up to the choice of the eigenfunction $v(t)$.  (Any constant multiple of 
$v$ is a parameterization of the stable normal bundle, but up to this choice of scaling
the solution is unique).  If the solution $P$ exists, it is as regular as $g$.  In this 
case $P$ is real analytic if $g$ is \cite{MR1976080, MR2177465}.
Moreover, in the case of one stable exponent,
there exists a choice of scaling small enough that the solution $P$ exists, and 
is analytic.  See \cite{MR1976079, MR2177465, MR3118249}.

\begin{figure}
    \centering
    \subfigure{{\includegraphics[width=.6\textwidth]{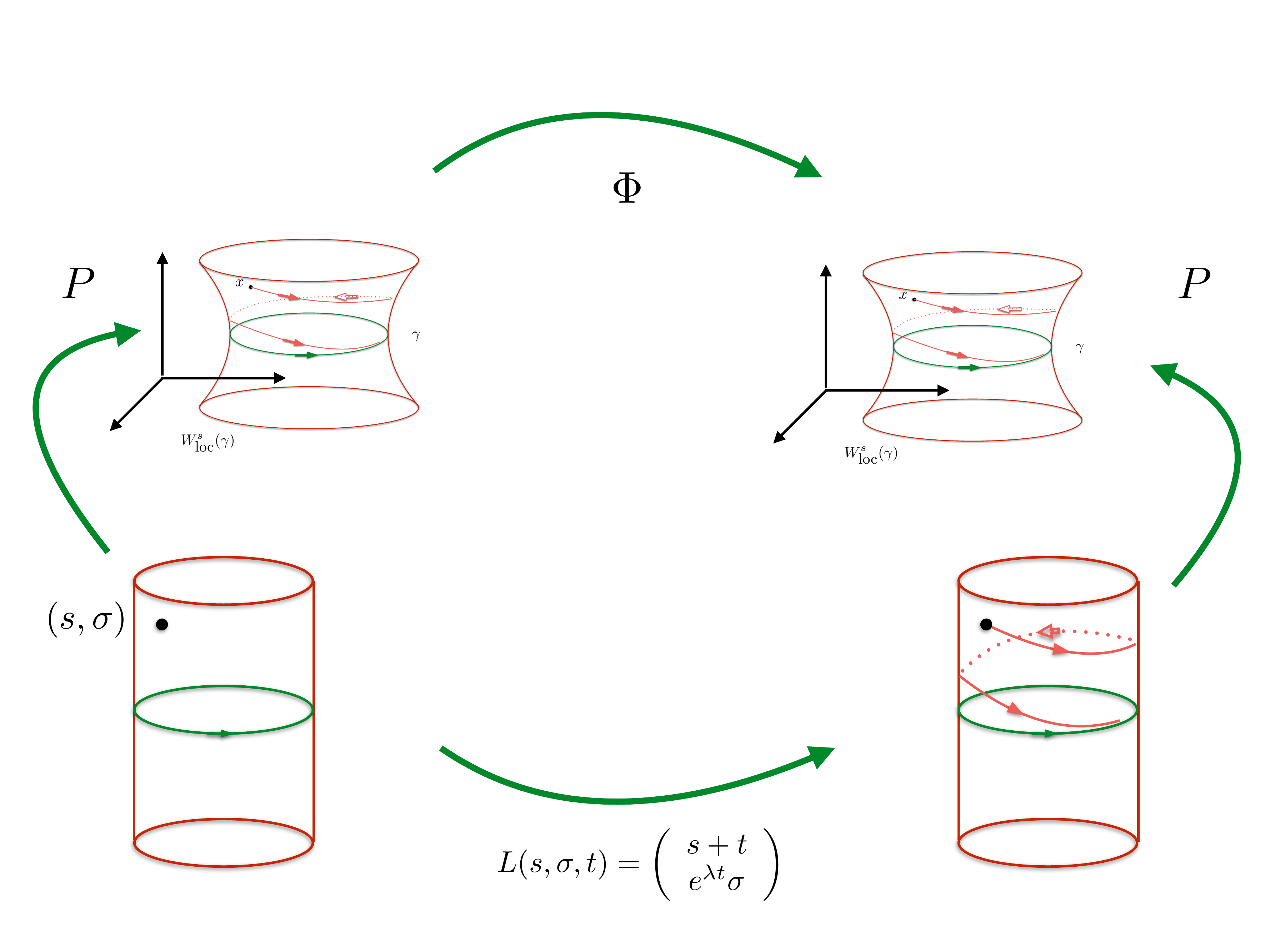} }}
\caption{Flow conjugacy generated by Equation \eqref{eq:InvarianceEquation}: 
Suppose that $P$ solves the invariance equation \eqref{eq:InvarianceEquation}.  Then 
the push forward of the model dynamics match the vector field 
induced on the image of $P$ by $g$, hence the map $P$ takes orbits on the 
cylinder to orbits on the image of $P$, i.e. $P$ conjugates the flow on the 
cylinder to the flow generated by $g$ and the above diagram commutes.
Since the flow on the cylinder is known, 
we obtain the conjugacy given in Equation \eqref{eq:FlowConjugacy}. In particular, 
since $P$ maps the zero section in the cylinder to $\gamma$, and since all 
orbits in the cylinder accumulate to the zero section, the image of $P$ is 
a local stable manifold for $\gamma$.}
\label{fig:conjSketch}
\end{figure}

\subsection{Formal series and the reduction to homological equations}
Since $P$ exists and is analytic it makes sense to seek a power series
solution 
\begin{equation}\label{eq:ExpansionParametrization}
P(t,\sigma)= \sum_{\alpha=0}^\infty A_{\alpha}(t) \sigma^\alpha, 
\end{equation}
where $\sigma \in (-1,1)$ and the functions $A_\alpha \colon \mathbb{R} \to \mathbb{R}^M$ are
analytic and $T$-periodic.  
Plugging the power series into Equation \eqref{eq:InvarianceEquation}, expanding the 
nonlinearities, and matching like powers leads to equations for the unknown Taylor coefficient 
functions.

\bigskip

\noindent \textbf{Example: the Lorenz system} in 
the specific case of the Lorenz field 
we have that $M=3$, and that the vector field $g$ is given by 
\begin{equation*}
g(x,y,z) =
\left( \begin{array}{c}
\sigma (y-x) \\
x\rho -x z - y \\
x y -\beta z 
\end{array} \right),
\end{equation*}
where $\beta,\rho,\sigma$ are positive constants.
Suppose that $\gamma(t) =  (x(t), y(t), z(t))$ is an analytic period- $T$ 
orbit with stable (or unstable) exponent $\lambda$
and that 
$v(t) = (v_1(t), v_2(t), v_3(t))$ is an analytic $T$-periodic parameterization of 
the stable normal bundle.  
We look for 
\[
P(t, s) = \sum_{\alpha = 0}^\infty A_\alpha(t) s^\alpha = 
\sum_{\alpha = 0}^\infty \left(
\begin{array}{c}
a^{(1, \alpha)}(t) \\
a^{(2, \alpha)}(t) \\
a^{(3, \alpha)}(t)
\end{array}
\right) s^\alpha
\]
Then 
\[
A_0(t) = \gamma(t), \quad \quad \quad \mbox{and} \quad \quad \quad 
A_1(t) = v(t).
\]
Since 
\[
 \frac{\partial}{\partial t}P(t, s) + \lambda s \frac{\partial}{\partial s}P(t, s) =
 \sum_{\alpha = 0}^\infty \left(\frac{d}{dt} A_{\alpha}(t) + \alpha \lambda A_{\alpha}(t) \right) s^\alpha,
\]
and 
\begin{align*}
g(P(t,s))  &= \sum_{\alpha = 0}^\infty 
\left(
\begin{array}{c}
\sigma \left( a^{(2, \alpha)} - a^{(1, \alpha)} \right) \\
\rho a^{(1, \alpha)}  - a^{(2, \alpha)} - \sum_{\beta = 0}^\alpha a^{(1, \alpha- \beta)} a^{(3, \beta)} \\
- \beta a^{(3, \alpha)} +     \sum_{\beta = 0}^\alpha a^{(1, \alpha- \beta)} a^{(2, \beta)} 
\end{array}
\right) s^\alpha,
\end{align*}
equating and matching like powers of $s$ leads to 
\begin{align*}
\frac{d}{dt} A_{\alpha}(t) + \alpha \lambda A_{\alpha}(t)  &=
\left(
\begin{array}{c}
\sigma \left( a^{(2, \alpha)} - a^{(1, \alpha)} \right) \\
\rho a^{(1, \alpha)}  - a^{(2, \alpha)} - \sum_{\beta = 0}^\alpha a^{(1, \alpha- \beta)} a^{(3, \beta)} \\
- \beta a^{(3, \alpha)} +     \sum_{\beta = 0}^\alpha a^{(1, \alpha- \beta)} a^{(2, \beta)} 
\end{array}
\right)
\\
&= \left(
\begin{array}{c}
\sigma \left( a^{(2, \alpha)} - a^{(1, \alpha)} \right) \\
\rho a^{(1, \alpha)}  - a^{(2, \alpha)} - a^{(1, \alpha)} a^{(3, 0)} - a^{(3, \alpha)} a^{(1, 0)} \\
- \beta a^{(3, \alpha)} +  a^{(1, \alpha)} a^{(2, 0)} + a^{(2, \alpha)} a^{(1, 0)}
\end{array}
\right) 
+
\left(
\begin{array}{c}
0  \\
-\sum_{\beta = 1}^{\alpha-1} a^{(1, \alpha- \beta)} a^{(3, \beta)} \\
\sum_{\beta = 1}^{\alpha-1} a^{(1, \alpha- \beta)} a^{(2, \beta)} 
\end{array}
\right).
\end{align*}
Noting that 
\begin{align*}
 \left(
\begin{array}{c}
\sigma \left( a^{(2, \alpha)} - a^{(1, \alpha)} \right) \\
\rho a^{(1, \alpha)}  - a^{(2, \alpha)} - a^{(1, \alpha)} a^{(3, 0)} - a^{(3, \alpha)} a^{(1, 0)} \\
- \beta a^{(3, \alpha)} +  a^{(1, \alpha)} a^{(2, 0)} + a^{(2, \alpha)} a^{(1, 0)}
\end{array}
\right) &= 
\left[
\begin{array}{ccc}
-\sigma & \sigma & 0 \\
\rho  - z(t) & 1  & - x(t) \\
y(t)  & x(t) & -\beta  
\end{array}
\right] A_{\alpha}(t) \\
&= Dg(\gamma(t)) A_\alpha(t),
\end{align*}
define the functions $g_\alpha$ by 
\[
g_{\alpha}(A_1(t), \ldots, A_{\alpha-1}(t)) := 
\left(
\begin{array}{c}
0  \\
-\sum_{\beta = 1}^{\alpha-1} a^{(1, \alpha- \beta)}(t) a^{(3, \beta)}(t) \\
\sum_{\beta = 1}^{\alpha-1} a^{(1, \alpha- \beta)}(t) a^{(2, \beta)}(t)
\end{array}
\right).
\]
We write $g_\alpha(P)$ for short.  
Now we seek $A_\alpha(t)$ the $T$-periodic solution of the equation
\begin{equation} \label{eq:homEqGeneral}
\frac{d}{dt} A_{\alpha}(t)   - \left(Dg(\gamma(t)) -  \lambda \alpha \mbox{Id}\right)  A_{\alpha}(t)  = g_{\alpha}(P),
\end{equation}
which we refer to as the homological equation for $P$.
Note that this is a linear inhomogeneous first order ordinary differential equation with 
periodic coefficients, and that the right hand side is independent of $A_\alpha$.
Indeed, $g_\alpha$ depends recursively on lower order terms.
The Floquet theory guarantees that our homological equation has a unique periodic solution for each 
$\alpha \geq 2$.  Then we recursively solve the equations to order $N$
and have the approximate solution
\[
P^N(t, s) = \sum_{\alpha = 0}^N A_\alpha(t) s^\alpha.
\]

\begin{remark}[A-posteriori error analysis] \label{rem:errors} {\em
Truncation error analysis is treated carefully in \cite{robertoII}.  Note that 
the analysis in that reference is independent of the basis used to represent $A_\alpha(t)$: 
only the implementation exploits that these functions are given as Fourier series.
Then the methods of the work just cited apply directly to the expansions used in the present work. 
The key to the analysis in \cite{robertoII} is that the approximation $P^N$ have small defect.
In the present work we only check the defect numerically, and postpone to an 
upcoming work more careful analysis of the errors for our Chebyshev-Taylor approximations.  }
\end{remark}

\bigskip

The calculations above generalizes to any polynomial vector field in the obvious way, and we have that 
the $A_\alpha$ satisfy homological equations of exactly the form 
given in Equation \eqref{eq:homEqGeneral}.  Only 
the term $Dg(\gamma(t))$ and
the form of the recursive functions $g_\alpha(P)$
depend explicitly on the form of the vector field $g$.
In the other examples in the present work we simply write down the correct homological 
equations and leave the derivations as an exercise for the interested reader.

\begin{figure}
    \centering
    \subfigure{{\includegraphics[width=.35\textwidth]{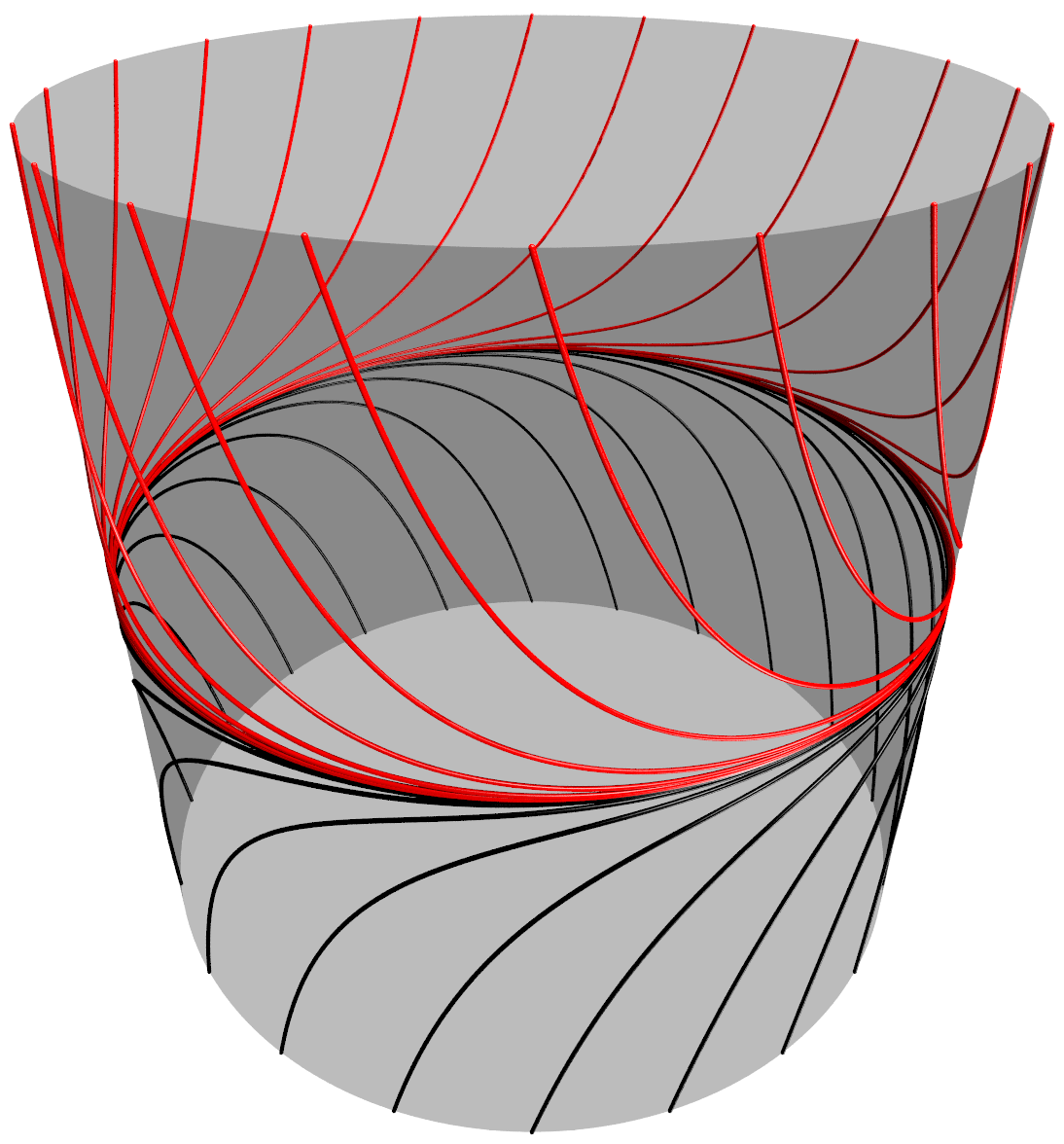} }}
    \subfigure{{\includegraphics[width=.60\textwidth]{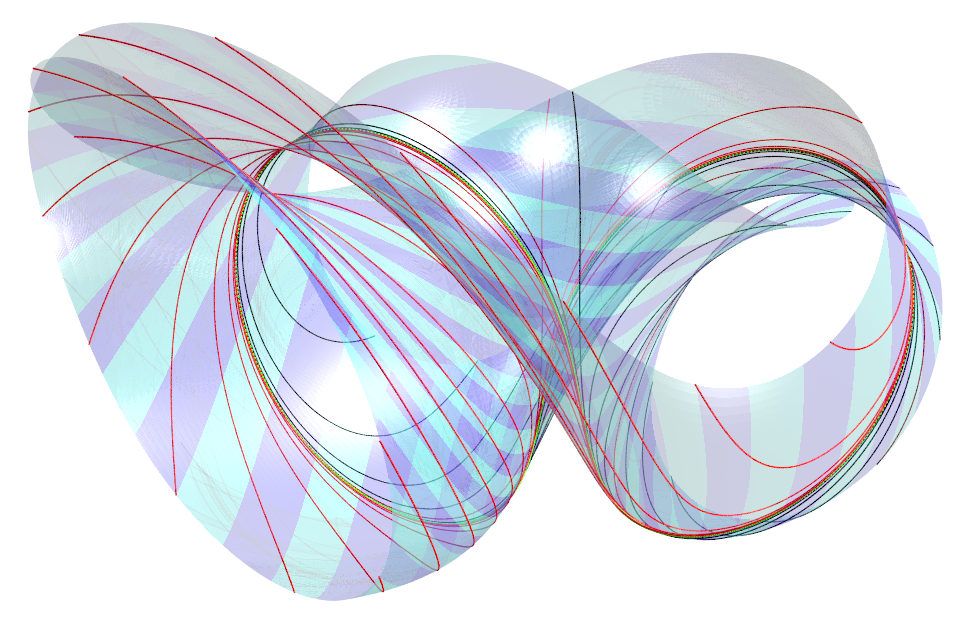} }}
    \caption{The Periodic orbit AB of the Lorenz system with classical parameters
     (green) and a Parametrized local stable manifold (light and dark blue).
     Each color represent a different subdomain 
     for the Chebyshev series.  The conjugacy describing the dynamics is illustrated by
    flowing a number of orbits near the boundary of the cylinder. 
     When these orbits are mapped by $P$ to the phase 
     space, we see the more complicated and nonuniform embedding of the 
     parameterization.  Compare to the schematic given in Figure \ref{fig:conjSketch}.
      }\label{fig:ABstableManifoldWithOrbits}
\end{figure}

%
%
%

\subsection{A dynamical perspective on automatic differentiation} \label{sec:autoDiff}
The term automatic differentiation refers to a whole suite of methods 
for managing the complexity of problems involving high order 
derivatives.  These methods exploit the fact that differentiation and multiplication
are related through the chain rule.  
For a general introduction and an overview of the literature 
we refer to \cite{MR1431038, MR2204531}, though in the present work our 
use of the term is much closer to that of Chapter $2$ of 
\cite{MR3467671} and also \cite{MR2146523}.  See also Chapter  $4.7$ of 
\cite{MR3077153}.  The discussion of the literature in this last reference is 
especially illuminating, though the reference does not use the 
``automatic differentiation'' terminology.   

While automatic differentiation can be viewed as a technique for computing 
transcendental functions of polynomials (where the result should again be expressed as a 
polynomial), automatic differentiation can also bee seen as a 
technique transforming non-polynomial into polynomial problems by appending
additional polynomial differential equations.  The solutions of the 
appended differential equations are required to give 
the non-polynomial terms of the original problem.  
When viewed this way it is possible to use automatic differentiation to solve
problems involving Fourier and Chebyshev rather than just Taylor bases.
Such extensions are discussed for example in \cite{Ransford}, where 
automatic differentiation is simply viewed as a change of variables which 
transforms the given problem to a polynomial problem.  

In general a change of variables will disturb the orbit structure of a differential equation, 
unless the change of variables induces a dynamical (semi)-conjugacy.
This issue is considered at length in \cite{shaneJay}, 
where a dynamical systems interpretation of automatic differentiation is introduced.
We adopt this interpretation here as well.  So, 
suppose that $U \subset \mathbb{R}^d$ is an open set and that 
$f \colon U \to \mathbb{R}^d$ is a real analytic vector field with non-polynomial nonlinearity.
From the perspective of the present work, the goal of automatic differentiation 
is to find a mapping $R \colon U \to \mathbb{R}^D$ with $D > d$ and a polynomial 
vector field $F \colon \mathbb{R}^D \to \mathbb{R}^D$
having that 
\[
\pi_d(R(\mathbf{x})) = \mathbf{x}, 
\]
where $\pi_d \colon \mathbb{R}^D \to \mathbb{R}^d$ is the 
projection into the lower dimensional space, and that 
\begin{equation}\label{eq:autoDiff}
DR(\mathbf{x}) f(\mathbf{x}) = F(R(\mathbf{x})), 
\quad \quad \quad \quad \mbox{for all } \mathbf{x} \in U.
\end{equation}
This leads to the following observations.
\begin{itemize}
\item Since $\pi_d R = \mbox{Id}_d$, it follows that 
\[
\pi_d F(R(\mathbf{x})) = f(\mathbf{x}),
\]
so that the original field is recovered by projection.  
\item Equation \eqref{eq:autoDiff} is an infinitesimal conjugacy, and 
implies that $R$ maps orbits of $f$ to orbits of $F$.
Moreover it says that graph of $R$ is invariant under the flow generated by $F$. 
\item Then when we solve the differential equations given by $F$ with initial 
conditions on the graph of $R$, we can recover the orbits of $f$ simply by projecting.
\end{itemize}
The procedure for choosing $R$ and $F$ is best illustrated through examples,
but it is worth noting that $R$ will have the same domain as $f$.  So even though 
$F$ is polynomial and hence entire, the composition of $F \circ R$ will have the 
same singularities as $f$.

\medskip

\noindent \textbf{Example:}
Consider the Kepler problem 
\begin{align*}
x' &= y \\
y' &= -\frac{M}{|x|^3} x.
\end{align*}
We introduce the new variable 
\[
z = \frac{1}{|x|} = \frac{1}{\sqrt{x^2}},
\]
and note that away from the origin one has 
\[
z' = \frac{d}{dt} (x^2)^{-1/2} = \frac{-1}{2} \left(x^2 \right)^{-3/2} 2x' = - z^3 y.
\]
Then taking 
\[
R(x,y) = \left(
\begin{array}{c}
x \\
y \\
\frac{1}{|x|}
\end{array}
\right),
\quad \quad \quad \mbox{and} \quad \quad \quad 
F(x,y,z) = \left(
\begin{array}{c}
y \\
-  M z^3 x \\
- z^3 y
\end{array}
\right),
\]
one easily checks the conjugacy Equation \eqref{eq:autoDiff}, 
and also that the first two components of $F \circ R$ recover $f$.
Solution curves of the vector field $F$ with initial conditions on the graph of 
$R$ recover solutions of the Kepler problem as long as they remain on the graph of $R$. 
That is, if 
\[
\left(
\begin{array}{c}
x' \\
y' \\
z'
\end{array}
\right) = F(x,y,z),
\quad \quad \quad \mbox{and } 
\left(
\begin{array}{c}
x(0) \\
y(0) \\
z(0)
\end{array}
\right) = R(x_0, y_0),
\]
then $x(t), y(t)$ is a solution of the Kepler problems as long as $x(t), y(t), z(t)$
is on the graph of $R$.  Moreover, the trajectory leaves the graph of $R$ if 
and only if there is a collision in the the Kepler problem.  This example is 
typical of automatic differentiation for problems in celestial mechanics.  
A thorough discussion -- from the dynamical systems point of view --
of automatic differentiation for the circular restricted four body  
problem is found in  \cite{shaneJay}.

\section{Chebyshev Expansion for Periodic Solutions of $\dot u(t)=h(u(t),t)$} \label{Sec:ChebyshevExpansion}
The periodic orbit $\gamma$, 
its stable/unstable normal bundle $v(t)$, and the higher order 
Taylor coefficients $A_\alpha(t)$ for $2 \leq \alpha \leq N$ are all periodic solutions of 
non-autonomous differential equations
of the form
\begin{align*}
\dot{u}(t)= h(u(t), t),
\end{align*}
with period $\tau$, and where $h:\R^M \times \R \to \R^M$. In each case 
we look for a solutions $v: [0, \tau] \to \R^M$ satisfying the boundary value problem
\begin{align}\label{eq:BVPperiodicOrbit}
\begin{cases}\dot{v}(t) = h(v(t),t), \\
v(0)=v(\tau).
\end{cases}
\end{align}
We develop a Chebyshev scheme to solve this class of problems.

To begin, break the solution into sub-pieces using the following mesh. Let $0=t_0<t_1<t_2<\hdots < t_D= \tau$,
and for $i=1,\hdots,D$ define $v^i(t)$ on $[t_{i-1},t_i]$ such that
\[
v^i(t)=v(t).
\]
Thus, each $v^i(t)$ is a solution of $\dot u(t)=h(u(t), t)$. Moreover, at each of the points $t_i$, $1\leq i \leq D-1$, 
two different pieces are defined and they must agree.  Then we impose the
boundary conditions
\begin{align}\label{eq:BVPsubdomain}
\begin{cases}
 \dot v^i(t) = h(v^i(t),t), & \\
 v^i(t_{i-1})= v^{i-1}(t_{i-1}), 
\end{cases}
\end{align}
for $i\neq 1$, and to impose periodicity 
\begin{align*}
\begin{cases}
 \dot v^1(t) = h(v^1(t),t), & \\
 v^{1}(t_{0})= v^{D}(t_D).
\end{cases}
\end{align*}

We want to expand each piece $v^i(t)$ using Chebyshev polynomials. To do so, we first rescale 
the problem to the interval $[-1,1]$. First, note that since there are no time dependence in 
$g$ (the original system for which we compute the periodic orbit) we can translate the time 
domain to some interval $[-L_i,L_i]$ and then rescale time $t \mapsto \frac{t}{L_i}$, so that 
the solution satisfies
\begin{align}\label{eq:ODEforSubPiece}
\dot{v}^i(t) = L_i h(v^i(t),t),
\end{align}
for all $i=1,\hdots,D$.

For the time varying case, i.e. the case of the bundles and the homological equations,
the solution will depend on the lower power $\alpha$ of the coefficient in 
question.  Since the bundles and Taylor coefficients of the parameterization 
are all periodic with \textit{the same period}, we choose a fixed mesh for all the 
problems. That is, the number of subdomains $D$ will be the same at every step. 
Moreover each $L_i$ is a fixed proportion of the global period. 
So $L_i=p_iL$ for $p_i\in(0,1]$ some given constant. 
In applications unless otherwise specified we use a uniform mesh, so that $p_i=\frac{1}{D}$ for all $i$.

We introduce a Chebyshev expansion for each sub-piece $v^i$ that are now defined on 
$[-1,1]$. Let $v^{(i,j)}:[-1,1] \to \R$ denote the $j-$th component of $v^i$ for all $i=1,\hdots,M$. 
For any $i=1,\hdots,D$ and $j=1,\hdots,M$, we set
\begin{align} \label{eq:ChebyExpansion}
  v^{(i,j)}(t) \bydef a_0^{(i,j)} + 2\sum_{k=1}^\infty a_k^{(i,j)} T_k(t), 
\end{align}
where $T_k$ is the $k-$th Chebyshev polynomial, which are defined as follows.
%
%
%
%
%
%
\begin{definition}
  The Chebyshev polynomials $T_k : [-1,1] \rightarrow \mathbb{R}$, $k=0,1,2,\dots$ are defined by $T_0(t)=1$ and 
  $T_1(t)=t$ and the recurrence relation
  \begin{equation*}
   T_{k+1}(t) = 2tT_k(t)-T_{k-1}(t), \quad k \geq 1.
  \end{equation*}
It is well known that these polynomials satisfy $T_k(\cos\theta) = \cos(k\theta)$, 
a fact which can be used to prove further results relating the 
Chebyshev series to results from Fourier analysis.
\end{definition}

\begin{definition}\label{def:ChebyshevUnknown}
We denote the set of unknown Chebyshev coefficients of the full periodic orbit $v$ by $$A=\left( A^{(1)},A^{(2)},\hdots,A^{(D)}\right).$$ So that for every $i=1,2\hdots,M$, $A^{(i)}$ represent the set of coefficients of the $i-$th piece of $v$ defined on the interval $[t_{i-1},t_i]$. Moreover, for a fixed $k$, $A_k^{(i)} \in \R^M$ since each has an image included in $\R^M$, which are all expanded using distinct Chebyshev expansions. Finally, $a_k^{(i,j)}$ denotes the $k-$th coefficient of the $j-$th dimension of $v^i$.
\end{definition}

\noindent To rewrite the system as an operator defined on the set $A$
we integrate \eqref{eq:ODEforSubPiece} from $-1$ to $t$ and obtain
\begin{align}\label{eq:IntegralForm}
v^i(t)-v^i(-1)= L_i\int_{-1}^t h(v^i(s),s) ds.
\end{align}

Note that $h:[-1,1] \to \R^M$, since it depends on the lower and current term of the expansion of the parameterization which all are defined on $[-1,1]$ after rescaling time. Thus it can also be expanded using Chebyshev polynomials. We set
\[
h^j(v^i(t),t)= c_0^{(i,j)} + 2\sum_{k=1}^\infty c_k^{(i,j)} T_k(t), 
\]
and substitute both Chebyshev expansions in \eqref{eq:IntegralForm} to get an equation whose only time dependence is in the Chebyshev polynomials themselves/the integral. We use the recurrence formulas for the integral of the Chebyshev polynomials and rewrite the initial condition so that after simplification the Chebyshev coefficients need to satisfy a set of conditions defined in the space of Chebyshev coefficients. That is, for all $i=1,\hdots,D$, $j=1,\hdots,M$ and $k\geq 0$, we 
define
\[
f_k^{(i,j)}(L,A)=0.
\]
Each $f_k^{(i,j)}(L,A)$ is given by
\begin{align} \label{eq:ChebyProblem}
f_k^{(i,j)}(L,A) \bydef \begin{cases}
\displaystyle \left(a_{0}^{(i-1,j)} + 2\sum_{l=1}^\infty a_{l}^{(i-1,j)}\right) -\left( a_{0}^{(i,j)} + 2\sum_{l=1}^\infty a_{l}^{(i,j)}(-1)^l \right), &\mbox{if}\; k=0 \\
2ka_{k}^{(i,j)} + L_i c_{k\pm 1}^{(i,j)}, &\mbox{if}\; k>0.
\end{cases}
\end{align}
\begin{definition}
We write $c_{k\pm1}$ to denote $c_{k\pm1} := c_{k+1}-c_{k-1}$. 
\end{definition}
\noindent We omit the derivation of the operators $f_k^{(i,j)}$. 
For further details, and to see why the term $k=0$ does not depend 
on the vector field but only on the boundary condition, 
we refer to \cite{LessardReinhardt}.

The boundary condition provides the initial component for the new problem, that is for 
the case $f_0^{(i,j)}$. So we use the form given in Equation \eqref{eq:ChebyProblem} 
for every subintervals except for the case $i=1$, for which we use
\begin{align} \label{eq:periodicity}
f_0^{(1,j)}(L,A) =\left(a_{0}^{(D,j)} + 2\sum_{l=1}^\infty a_{l}^{(D,j)} \right) -
\left( a_{0}^{(1,j)} + 2\sum_{l=1}^\infty a_{l}^{(1,j)}(-1)^l\right).
\end{align}

The operator involves coefficients of the Chebyshev expansion of $h$, 
which we need to write in terms of the unknowns $A$. Since the Chebyshev expansion of a sum (or difference) of two functions $f$ and $g$ will be given by the sum (or difference) of the Chebyshev coefficients of $f$ and $g$, and since we assume $h$ to be polynomial, the only other case to handle is a product. We use the following Lemma.
\begin{lemma}
If $f(t)$ and $g(t)$ are expanded with Chebyshev series so that
\[
f(t)= a_0 + 2\sum_{k=1}^\infty a_k T_k(t) \quad \mbox{and} \quad g(t)= b_0 + 2\sum_{k=1}^\infty b_k T_k(t),
\]
then $f(t)g(t)= (a\ast b)_0 + 2 \displaystyle \sum_{k=1}^\infty (a\ast b)_k T_k(t)$. 
Here $\ast$ denotes the discrete convolution product
\[
(a\ast b)_k = \sum_{\substack{k_1+k_2= k \\ k_1,k_2 \in \Z}} a_{|k_1|}b_{|k_2|}.
\]
\end{lemma}
Therefore, we can completely rewrite each $f_k^{(i,j)}$ without the coefficients $c_k$. To compute an approximation of the solution, we consider a finite dimension approximation 
\begin{align*}
F= \left\lbrace f_k^{(i,j)} : 1\leq i \leq D, 1 \leq j \leq M, 0 \leq k \leq m-1 \right\rbrace,
\end{align*}
so that $F:\R^{mMD} \to \R^{mMD}$. Note that the boundary conditions introduce a dependence between each subdomain, hence we solve 
the equations simultaneously. In the case $\alpha\geq 2$ one can directly use Newton's method to find $\bar A$ such that $F(\bar A) \approx 0$, where $\bar A$ is a finite dimensional approximation of the unknowns $A$ with same dimension as $F$.

\bigskip

\noindent \textbf{The first order data:}
In the case $\alpha=0$ or $\alpha=1$, we have additional unknowns hence
we also need to add phase conditions.  For the periodic orbit $\gamma$, the period 
$\tau$ is an unknown.  For the eigenvalue equation defining the stable/unstable bundles
it is the eigenvalue/Floquet exponent which is unknown.  In both cases we must balance the 
equations.  



For the periodic orbit $\gamma(t)$, we replace $h(v(t),t)$ by the given vector field $g(v(t))$. 
Note that for any $s$, the time translated solution $\gamma_s(t)=\gamma(t+s)$ 
is still a solution of the problem since it is periodic and satisfies 
$\dot{\gamma_s}(t)= g(\gamma_s(t))$. In order to isolate a solution
we impose a Poincar\'{e} condition
\begin{align}
\dot{p}_0\cdot (p_0-\gamma(0))=0,
\end{align}
for some $p_0 \in \R^M$ with $\dot{p_0} :=g(p_0)$. 
This condition translates into a condition on $ v^1(-1)$ after applying the proper change of variable
to Chebyshev coefficients as
\begin{align}\label{eq:PoincareCondition}
\dot{p_0}\cdot p_0  -\sum_{j=1}^M \dot{v_0}^{(j)} \left(a_{0,0}^{(1,j)} +2\sum_{k=1}^\infty a_{0,k}^{(1,j)}(-1)^{k} \right)=0.
\end{align}

The normal bundle satisfies the eigenvalue problem
\[
\dot v(t) = Dg(\gamma(t))v(t) - \lambda v(t),
\]
where $Dg(\gamma(t))$ is the periodic matrix given by the derivative of $g$ evaluated at the periodic orbit. It follows that any rescaling $k v(t)$ of the bundle is again a solution associated to the same eigenvalue $\lambda$. To isolate a solution we fix $\|v(0)\| = K$, where we are free to 
choose $K$.   In term of the Chebyshev coefficients
\begin{align*}
\sum_{j=1}^M \left(a_0^{(1,j)} + 2\sum_{k=1}^\infty a_k^{(1,j)}(-1)^k\right)^2 - K = 0.
\end{align*}
For simplicity we truncate this condition to
\begin{align}\label{eq:eigenvectorCondition}
\sum_{j=1}^M \sum_{k=0}^{k_0} \left(a_k^{(1,j)}\right)^2 - K = 0,
\end{align}
which still isolates an eigenfunction.

\section{Examples}
We introduce the following operator to simplify the expansion of the operator $F$ in each example.
\begin{definition}
Let $(j_1,\hdots,j_n)$ be a set consisting of hyperscript corresponding to component of the solution. That is $1\leq j_k \leq M$ for all $k=1,\hdots,n$, we denote their Cauchy product of convolutions by
\[
\mathcal{C}_{\alpha,k}^{(i)}(j_1,\hdots,j_n)= \sum_{\substack{\alpha_1+\hdots+\alpha_n=\alpha \\ \alpha_j \in \Z^+ }} \sum_{\substack{k_1+\hdots+k_n=k \\ k_j \in \Z }} a_{\alpha_1,k_1}^{(i,j_1)}a_{\alpha_2,k_2}^{(i,j_2)}\cdots a_{\alpha_n,k_n}^{(i,j_n)}.
\]
Note that the case $\alpha=0$ simply returns the convolution product.
\end{definition}


\subsection{The example of the Lorenz system}
All of our numerical computations use the classical 
parameter values $\beta=\frac{8}{3}$, $\rho =27$ and $\sigma=10$. 
The operator defining the unknowns $A_\alpha(t)$ is given by
\begin{align*}
f^{(i,1)}_{\alpha,k}(A_\alpha) &= 2k a_{\alpha,k}^{(i,1)} +L_i\bigg(-\lambda\alpha a_{\alpha,k\pm 1}^{(i,1)} +\sigma( a_{\alpha,k\pm 1}^{(i,2)}-a_{\alpha,k\pm 1}^{(i,1)}) \bigg) \\
f^{(i,2)}_{\alpha,k}(A_\alpha) &= 2k a_{\alpha,k}^{(i,2)} +L_i\bigg(-\lambda\alpha a_{\alpha,k\pm 1}^{(i,2)} +\rho a_{\alpha,k\pm 1}^{(i,1)} -\cC(2,3) -a_{\alpha,k\pm 1}^{(i,2)} \bigg), \\
f^{(i,3)}_{\alpha,k}(A_\alpha) &= 2k a_{\alpha,k}^{(i,3)} +L_i\bigg(-\lambda\alpha a_{\alpha,k\pm 1}^{(i,3)} +\cC(1,2) -\beta a_{\alpha,k\pm1}^{(i,3)} \bigg),
\end{align*}
with $k\geq 1$. The formula for $k=0$ is omitted since it is already explicitly given in 
\eqref{eq:ChebyProblem}. Note that $\mathcal{C}$ involves the lower order terms $A_\beta$ for $\beta\leq \alpha$ 
but we are only solving for $A_\alpha$ with the lower order terms fixed. 

Only the  $\alpha=0$ problem is nonlinear in the unknown Chebyshev coefficients
and this requires a good initial guess to obtain a periodic orbit. In the 
present work we use the data from \cite{Viswanath} as the input for a Newton method. 
For each $\alpha$ we truncate and consider the finite dimensional problem
\[
\hat A_\alpha = \left\lbrace a_{\alpha,k}^{(i,j)} : 0\leq k \leq m-1, 1\leq i \leq D, 1\leq j \leq 3 \right\rbrace.
\]
The corresponding truncated operator is such that $F:\R^{3mD +1} \to \R^{3mD+1}$ for $\alpha=0,1$ and $F:\R^{3mD} \to \R^{3mD}$ in the remaining cases. 

In Figure \ref{fig:stableLorenz} we computed the stable manifold for two periodic orbits. For these computations we use $D=10$ Chebyshev domains, 
$m=50$ Chebyshev coefficients per domain, and $N=50$ Taylor nodes. 
 The time rescaling in the operator (i.e. half the period) is $L\approx 1.1530$ for the shorter orbit and $L\approx 1.9101$ for the longer one. Each color represent a different subdomain of the Chebyshev time decomposition. 
\begin{figure}
 \begin{center}
  \subfigure{\includegraphics[width=0.44\textwidth]{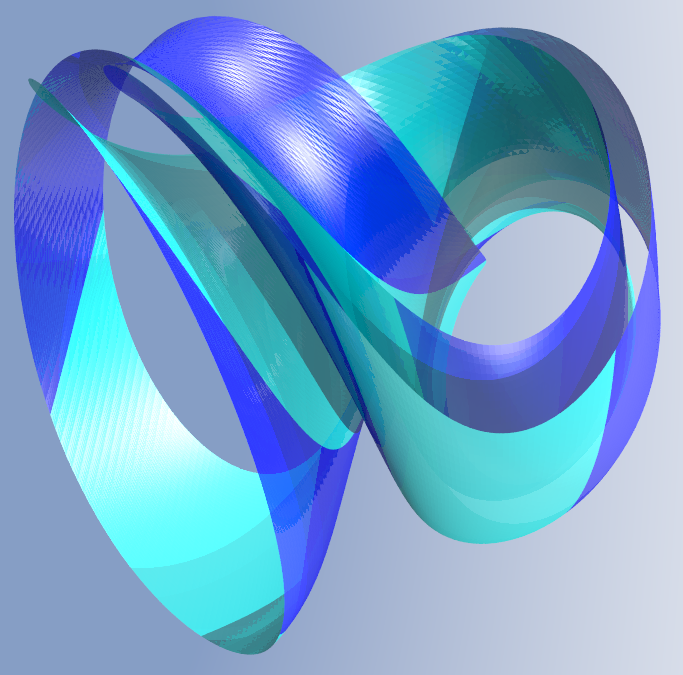}}
  \quad
  \subfigure{\includegraphics[width=0.52\textwidth]{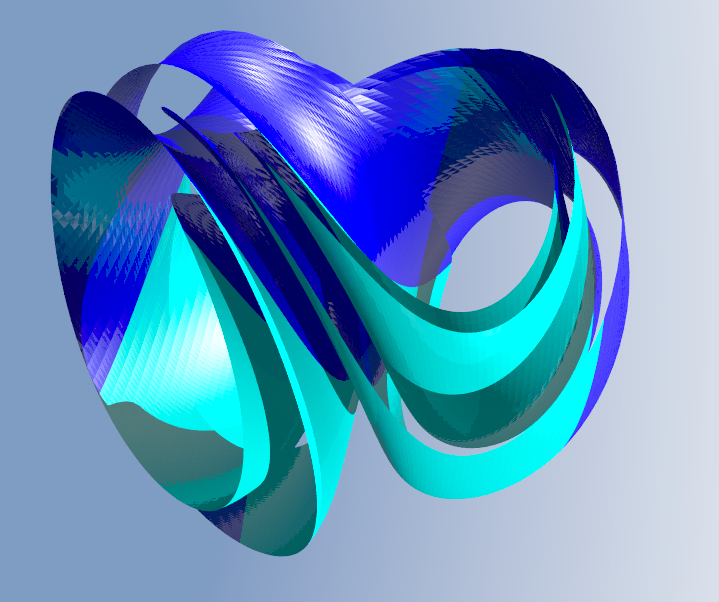}}
  \caption{Periodic orbits ABB (left) and AABBB (right) with their attached local stable manifold.} \label{fig:stableLorenz}
 \end{center}
\end{figure}

In Figure \ref{fig:AABBBcontinued} we extended the manifold on the right of figure \ref{fig:stableLorenz} by integrating backward in time $100$ points evenly distributed on the boundary of the parameterization. The orbits were computed by integrating backward in time for $t=0.5$. Using the conjugacy relation \eqref{eq:FlowConjugacy}, we have that one would need to integrate for $t=2.5212$ to go from $\sigma \approx 10^{-16}$ to the boundary $\sigma=1$. Thus, allowing the utilization of much smaller time lapse to get a good extension of the attractor for this orbit.
The discretization of the continued manifold is very coarse, but the example is only included to show
that nearly all the ``slow'' dynamics of the manifold is captured by the parameterization.  Once we
start integrating the local manifold orbits move away very rapidly.  
\begin{figure}
 \begin{center}
  \includegraphics[width=\textwidth]{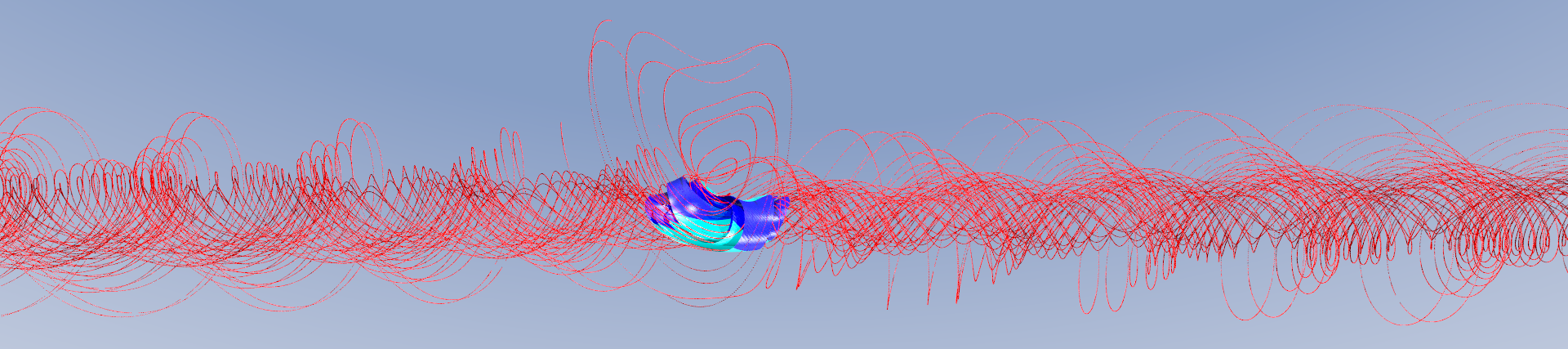}
  \caption{Periodic orbits AABBB and its stable manifold in blue with orbits in red. The manifold displayed in this figure is the same as in figure \ref{fig:stableLorenz}.} \label{fig:AABBBcontinued}
 \end{center}
\end{figure}

\bigskip

\noindent \textbf{A-posteriori error analysis:}
Since computing a local invariant manifold with the parameterization method 
does not necessarily involve a small parameter,  
its not always clear what it means to talk about ``the order'' of the 
error.  This is especially true when polynomials of high degree are used to 
approximate the manifold in a large neighborhood of the periodic orbit.
In this case it is more natural to use the notion of defect to quantify errors,
and we exploit the fact that the parameterization satisfies a conjugacy equation.

That is, since the equation \eqref{eq:FlowConjugacy} must be satisfied, we define
\[
Err(t_0) := \sup_{s \in [0, \tau]} \sup_{\sigma \in [-1,1]} \left\| P(s+t_0,e^{\lambda t_0}\sigma) - \Phi(P(s,\sigma),t_0)  \right\|,
\]
and  $t_0$ is some fixed test time. Sampling points in $[0, \tau]$ with $\sigma = \pm 1$ leads to a 
useful and numerically accessible estimate.

Some heuristics are also helpful.  For example we find that choosing $N$ and $K$ 
so that the norm of the last Taylor component is around machine precision 
leads to excellent results.  A useful norm for making this assessment is 
\[
\|a_N\| = \max_{i=1,\hdots,D}  \sum_{j=1}^3 \sum_{k=0}^{m-1} \left| a_{N,k}^{(i,j)} \right|,
\]
as this involves only sums of the known coefficients.

We test the conjugacy for $t_0= 10^{-5}$ and $t_0=1$ with $200$ different starting points evenly distributed on the parameterization of the manifold and then took the average of the resulting errors. The results are displayed in table \ref{table:ErrorTests}. From this table, one can note that the choice of $N$ and $K$ are such that the conjugacy error remains considerably small for longer period of time while the norm of the last Taylor sequence is not too far beyond machine precision.
\begin{table}
\begin{center}
\begin{tabular}{ |c|c|c| }
\hline
 $N$ & $t_0= 10^{-5}$ & $t_0=1$  \\
\hline
$20$ &  $9.1859\times 10^{-4}$ & $4.9280\times 10^{-1}$ \\
$40$ &  $4.2205\times 10^{-5}$ & $1.2568\times 10^{-2}$ \\
$60$ &  $1.2725\times 10^{-6}$ & $2.6406\times 10^{-4}$ \\
$80$ &  $3.1854\times 10^{-8}$ & $8.0463\times 10^{-6}$ \\
$100$ & $7.2073\times 10^{-10}$ & $3.404\times 10^{-8}$ \\
\hline
\end{tabular}
\quad \quad
\begin{tabular}{ |c|c|c| }
\hline
 $K$ & $\|a_{100}\|$  \\
\hline
$50$  & $7.4728\times 10^{-55}$  \\
$100$ & $8.4136\times 10^{-40}$ \\
$150$ & $5.3647\times 10^{-31}$ \\
$200$ & $9.4729\times 10^{-25}$ \\
$250$ & $6.6372\times 10^{-20}$ \\
\hline
\end{tabular}
\end{center}
\caption{Error for the conjugacy test for different Taylor dimension when $m=10$, $D=50$, $K=250$ and $k_0=10$ (left). Norm of the last Taylor dimension for different choices of $K$ when $m=10$, $D=50$, $N=100$ and $k_0=10$ (right). }\label{table:ErrorTests}
\end{table}

\subsubsection{Short Connecting orbit} \label{sec:shortConnection}
Following the convention established in \cite{MR3207723} we say that there is a 
short connection from $\gamma_1$ to $\gamma_2$ if the local parameterization
of the unstable manifold of $\gamma_1$ intersects the local parameterization of the 
stable manifold of $\gamma_2$.  (Here we are assuming that $\gamma_{1,2}$ 
come equipped with some choice of local manifold parameterizations.  In fact it is more correct 
to say that the connection is short relative to these fixed parameterizations).
In this case one can establish the existence of a 
connection without the use of any numerical integration.  The fact this is possible 
provides another illustration of the fact that our manifold computations
are in some sense not local.

\begin{figure}
 \begin{center}
  \includegraphics[width=0.6\textwidth]{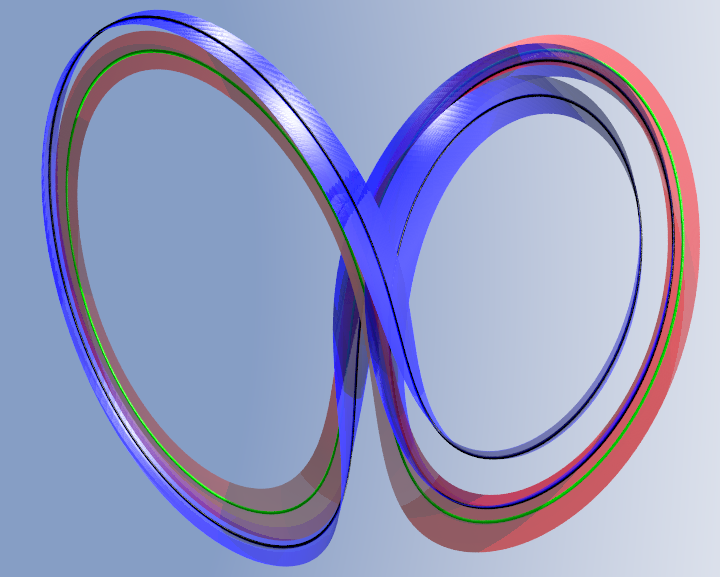} 
  \caption{Unstable manifold of the orbit AB (red) and stable manifold of the ABB orbit (blue). Both manifolds are computed using $m=100$, $D=6$, $N=10$, $K=10$ and $k_0=10$.
  One sees that the local manifolds intersect near the middle of the figure.} \label{fig:AB_ABB}
 \end{center}
\end{figure}

In Figure \ref{fig:AB_ABB} we display the stable manifold of the orbit AB and the unstable manifold of ABB
in Lorenz and observe several intersection of the two manifolds.
The boundary of the unstable manifold crosses the stable manifold
giving a connecting orbit. Let $P:[0,2L_1]\times[-1,1] \to \R^3$ denote the parameterization of the local stable manifold of the periodic orbit AB, whose period is $2L_1$. Similarly, let $Q:[0,2L_2]\times[-1,1] \to \R^3$ be the parameterization of the unstable manifold of the periodic orbit ABB, whose period is given by $2L_2$. Once the unstable manifold is restricted to one of its boundary circles only $3$ unknowns remain. We set $x=(\theta_s,\sigma_s,\theta_u)$ so that the desired intersection is a zero of
\[
S(x)= P(\theta_s,\sigma_s)-Q(\theta_u,1).
\]
Therefore $S:\R^3 \to \R^3$ and it is possible to apply Newton's method to obtain an approximation of the solution. Using this approach, we found
\[
x\approx (1.942170529091222, 0.000000560679355,1.253373698262391).
\]
Then the connecting orbit can be computed from those coordinates in the parameter space. Since the value $\sigma_s$ of $x$ is already quite small we used the conjugacy relation \eqref{eq:FlowConjugacy} with $t=2$ and it is already sufficient. 
For the other half of the connecting orbit, we use \eqref{eq:FlowConjugacy} again to determine the ``time of flight''. 
In this case, the unstable eigenvalue is $\lambda\approx 0.9947$, so one would need to integrate backward in time for $t \approx 34.7246$ to obtain $\sigma$ smaller than $10^{-15}$. 
Lets stress this point again: computing this orbit using the linear approximation of the stable/unstable manifolds would 
require almost $35$ units of time integration to cross from the unstable to the stable normal bundles. Using 
the parameterized manifolds no integration is necessary and the entire orbit is represented ``locally''.

These results are displayed in figure \ref{fig:FlowOfshortConnecting}. The green curve is the one obtained using the stable manifold while the trajectory in red correspond to the one using the unstable manifold and the time $t=34.7246$.
\begin{figure}
    \centering
    \subfigure{{\includegraphics[width=.35\textwidth]{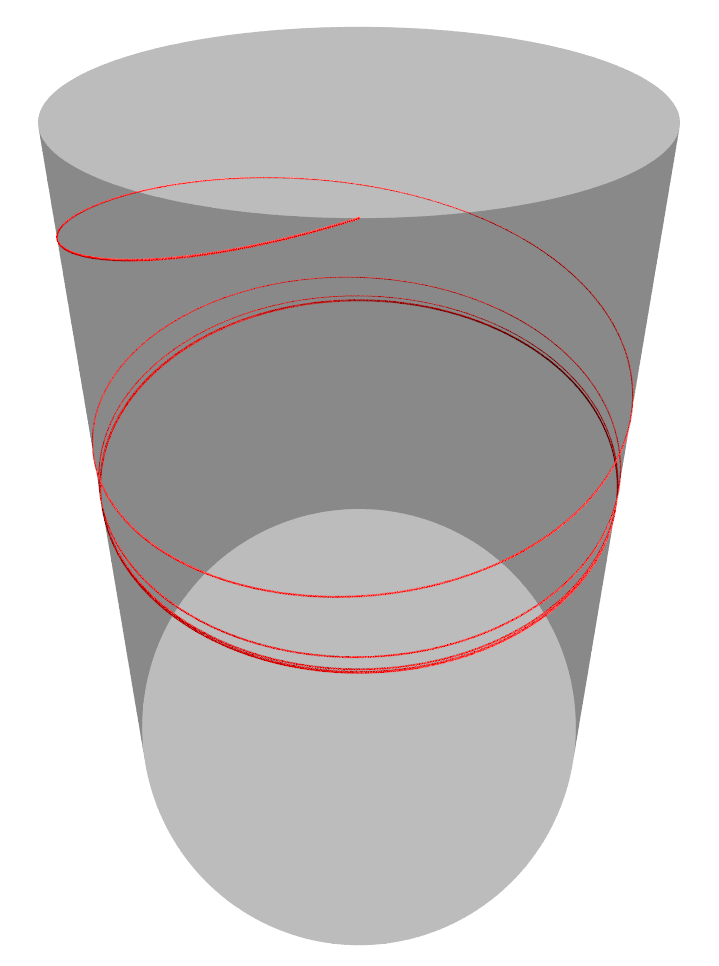} }}
    \subfigure{{\includegraphics[width=.60\textwidth]{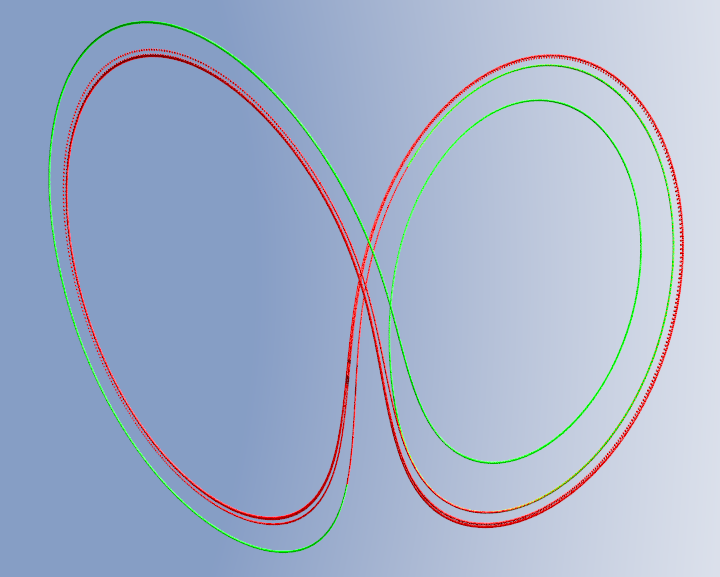} }}
    \caption{On the left, trajectory of the orbit in the parameter space of $Q$ under the linear flow given 
    by the Floquet exponent. On the right, image under $P$ of points for time $t\leq 2$ (green) and image under $Q$ of 
    the trajectory from the left (red).
    Compare to the schematic of Figure \ref{fig:conjSketch}.}\label{fig:FlowOfshortConnecting}
\end{figure}

\begin{figure}
    \centering
    \subfigure{{\includegraphics[width=.7\textwidth]{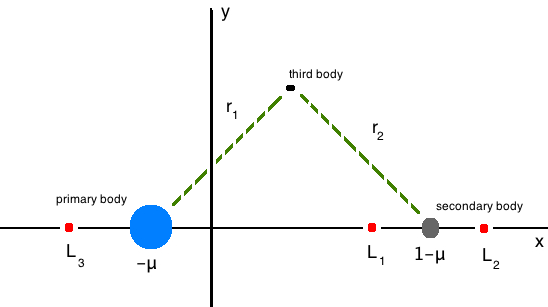} }}
\caption{Schematic representation of the CRTBP: two primary bodies with 
mass ratio $\mu$ move on circular orbits about their center of mass.    
After changing to a co-rotating frame 
the more massive primary is located at $x = -\mu, y = 0$ while
the smaller primary is fixed at $x = 1-\mu, y = 0$.  The CRTBP
studies the motion of a third and massless particle moving in the gravitational 
field of the primaries. The co-linear libration points $\mathcal{L}_j$, $j = 1,2,3$
are shown on the $x$ axis.  The libration points $\mathcal{L}_{4,5}$ form equilateral triangles
with the primaries but play no role in remainder of the discussion hence are not shown 
in the figure.}\label{fig:threeBodies}
\end{figure}

\subsection{The Circular Restricted Three Body Problem} \label{sec:CRTBP}
The circular restricted three body problem (CRTBP) describes the evolution of a 
massless particle moving in the gravitational field of two other massive bodies,  
called the primaries. It is assumed that the primaries move in circular orbits about 
their center of mass.  In this work we focus on the case that the third body moves
in the plane of the primaries. 
In co-rotating coordinates the equations of motion are
\begin{align}\label{eq:CRTBP_second_order_system}
 \begin{cases}
  x''=2y' + \frac{\partial H}{\partial x} \\
  y''=-2x' + \frac{\partial H}{\partial y}
 \end{cases}
\end{align}
with
\[
H(x,y)= \frac{1}{2}(x^2 +y^2) +\frac{(1-\mu)}{\sqrt{(x+\mu)^2 +y^2}} +\frac{\mu}{\sqrt{(x+\mu-1)^2 +y^2}}.
\]
Here $\mu$ is the mass ratio of the primary bodies. 
In the rotating frame:
\begin{enumerate}
\item the center of mass is at the origin,
\item the motion of the primaries is fixed and they sit on the $x-$axis
at $-\mu$ and $1-\mu$.
\end{enumerate}
This choice of coordinates introduces the  Coriolis effect, that is the 
coordinates are non-inertial.  

The CRTBP is much studied as an example in Hamiltonian dynamics and celestial mechanics
going back to the work of Poincar\'{e}.  It is a useful model of the motion of a satellite or astroid
influenced by a two body system such as Earth/Moon, Sun/Earth, or Sun/Jupiter.  It is also one of the 
simplest $N$-body systems which admits chaotic motions, hence is not integrable.   
For further discussion of this and more general $N-$body problems, we refer the reader to the 
books of \cite{MR2029316, MR2391999, MR1867240, MR1705705, MR1870302, MeyerHallOffin}.  See also 
\cite{MR2537575, MR3038565, MR649156, MR2163533, MR2475705, 
MR2086140, MR1878097, MR791842, MR590002, MR3284613, MR2114685}.
This list of references constitutes only the barest introduction to the relevant literature, 
but much more complete discussion is found in the books and papers just cited.

The system has five relative equilibrium points, also referred to as libration points, 
and we denote these by $\mathcal{L}_j$ with $j = 1,2,3,4,5$.  Three of these 
lie on the $x$-axis (the co-linear libration points).  
These are denoted $\mathcal{L}_{1,2,3}$.  See Figure \ref{fig:threeBodies}
for a schematic representation of the problem.

Each of the co-linear libration points are of saddle $\times$ center stability type.  
By an application of the Lyapunov center theorem, each of the centers 
gives rise to one parameter families of periodic orbits.
These are known as Lyapunov families.
We compute invariant manifolds attached to some Lyapunov orbits below.
First we transform to a polynomial problem.  

Rewriting \eqref{eq:CRTBP_second_order_system} as a vector field $\dot u = \hat{g}(u)$ gives 
\begin{equation}\label{eq:CRTBP_non_polynomial}
\hat g(u)=\begin{pmatrix}
 u^2 \\
 2u^4 +u^1 -\frac{(1-\mu)(u^1+\mu)}{\sqrt{(u^1+\mu)^2 + (u^3)^2}^3} -\frac{\mu (u^1-1+\mu)}{\sqrt{(u^1-1+\mu)^2 +(u^3)^2}^3} \\
 u^4 \\
 -2u^2 +u^3 -\frac{(1-\mu)u^3}{\sqrt{(u^1+\mu)^2 + (u^3)^2}^3} -\frac{\mu u^3}{\sqrt{(u^1-1+\mu)^2 + (u^3)^2}^3} \\
\end{pmatrix}.
\end{equation}
This system is non-polynomial and we propose a related polynomial system
through the use of automatic differentiation.  Following 
\cite{jpHalo, shaneJay, Ransford} and the discussion in 
Section \ref{sec:autoDiff} we 
derive the polynomial vector field
\begin{equation}\label{eq:g_CRTBP}
g(v)=\begin{pmatrix}
 v^2 \\
 2v^4 +v^1 -(1-\mu)(v^1+\mu)(v^5)^3 -\mu (v^6)^3(v^1-1+\mu) \\
 v^4 \\
 -2v^2 +v^3 -(1-\mu)v^3(v^5)^3 -\mu v^3(v^6)^3 \\
 -(v^5)^3((v^1+\mu)v^2 +v^3v^4) \\
 -(v^6)^3((v^1-1+\mu)v^2 +v^3v^4)
\end{pmatrix},
\end{equation}
with the additional (initial condition) constraints
\begin{align*}
 v^5(0) &= \frac{1}{\sqrt{(v^1(0) +\mu)^2 +(v^3(0))^2}} \\
 v^6(0) &= \frac{1}{\sqrt{(v^1(0) -1 +\mu)^2 +(v^3(0))^2}}. \\
\end{align*}
That is, we take $R$ defined by 
\[
R(u^1, u^2, u^3, u^4) = 
\left(
\begin{array}{c}
u^1 \\
u^2 \\
u^3 \\
u^4 \\
\frac{1}{\sqrt{ (u^1 +\mu)^2 + (u^3)^2}} \\
\frac{1}{\sqrt{(u^1-1+\mu)^2 + (u^3)^2}}
\end{array}
\right), 
\quad \quad \quad \mbox{and have} \quad 
DR(\mathbf{u}) \hat{g}(\mathbf{u}) = g(R(\mathbf{u})),
\]
where $\mathbf{u} = (u^1, u^2, u^3, u^4)$. 
Then the conjugacy of Equation \eqref{eq:autoDiff} is satisfied.
Moreover, one easily checks that 
$\pi_4(g(R(\mathbf{u}))) = \hat{g}(\mathbf{u})$.
Of course $R$ is defined only on the complement of the collision set of the CRTBP.

Since periodic orbits in the CRTBP occur in one parameter families parameterized by 
energy, we fix $L$ and look for a periodic orbit with this half period. 
Moreover, we use the well known reversible symmetry of the problem 
to formulate a different boundary condition that will still provide periodic 
solutions. For example, orbits in a Lyapunov family of one of the co-linear libration points
have no velocity in the $x$ direction when they cross the $x$ axis -- 
and the time between the two crossing of the $x$ axis are separated by exactly 
the half orbit. So that a periodic solution $\gamma$ 
starting on the $x$ axis with frequency $2L$ will satisfy
\begin{align} \label{eq:BoundaryCondCRTBP}
\gamma_2(0)= 0, \quad \gamma_3(0)=0, \quad \gamma_2\left(2L\right) 
= 0 \quad \mbox{and} \quad \gamma_3\left(2L\right)=0.
\end{align}
\begin{remark}
Note that the last two conditions of \eqref{eq:BoundaryCondCRTBP} could also be $\gamma_2(L)=\gamma_3(L)=0$. 
This would allow us to compute only the half orbit, the symmetry of the problem providing the solution for the second half of the trajectory. The use of Chebyshev expansion would allow such a choice. However, since
the periodic orbit is not the final product of our computations, but only an input
into the higher order equations we compute full orbits in this work.
\end{remark}
For the remaining two conditions, we simply rewrite the initial condition on $v^5$ and $v^6$ 
after imposing the symmetry, so that
\begin{align*}
 \gamma^5(0) = &\frac{1}{|\gamma^1(0) +\mu|}, \\
 \gamma^6(0) = &\frac{1}{|\gamma^1(0) -1 +\mu| }. \\
\end{align*}
Due to the choice of the boundary condition, time translation of the solution does not satisfy this system. 
Thus we drop the Poincar\'{e} condition previously given and obtain a system that is still fully determined and 
with isolated solutions (since we don't solve for the period we dispense with one scalar equation). 
Two possibilities arise from that remark, one could fix $L$ to a given value and still find an orbit as previously mentioned.
 The other choice is to use $L$ as a variable and obtain a fully determined system by adding a condition in which the energy level is fixed to a chosen constant. Such a choice is necessary for example if we want to compute heteroclinic connecting orbits.

We summarize the discussion in the following lemma.

\begin{lemma}\label{Lemma:periodicCRTBP}
Let $v^1(t),v^2(t),v^3(t),v^4(t)$ be periodic function with same period $\omega$ and such that $\dot v^1 = v^2$, 
$\dot v^3=v^4$. Let $v^5(t),v^6(t)$ satisfy
\begin{align} \label{eq:BVPv5}
\begin{cases}
 \dot v^5(t) = -(v^5(t))^3((v^1(t)+\mu)v^2(t) +v^3(t)v^4(t)), & \\
 v^5(0)= \frac{1}{\sqrt{(v^1(0) +\mu)^2 +v^3(0)^2}}, &
\end{cases}
\end{align}
and
\begin{align} \label{eq:BVPv6}
 \begin{cases}
\dot v^6(t) = -(v^6(t))^3((v^1(t)-1+\mu)v^2(t) +v^3(t)v^4(t)), & \\
v^6(0) = \frac{1}{\sqrt{(v^1(0) -1 +\mu)^2 +(v^3(0))^2}}. &
 \end{cases}
\end{align}
Then $v^5(t)$ and $v^6(t)$ are periodic with period $\omega$.
\end{lemma}
%
%
%
%
\begin{proof}
We first note that
\[
w^5(t)= \frac{1}{\sqrt{(v^1(t)+\mu)^2 +(v^3(t))^2}}
\]
is a solution of \eqref{eq:BVPv5}. But by unicity of the solution it follows that $v^5(t)=w^5(t)$. Moreover, we have that $w^5(t)$ is periodic with period $\omega$ since $v^1(t)$ and $v^3(t)$ are periodic with period $\omega$. Thus $v^5(t)$ is periodic with period $\omega$, as desired. For $v^6(t)$, the proof is similar using the fact that
\[
w^6(t)= \frac{1}{\sqrt{(v^1(t)-1+\mu)^2 +(v^3(t))^2}}
\]
is a solution of \eqref{eq:BVPv6}.
\end{proof}

To fix the energy we have to use an integral of the CRTBP, namely
\[
E(x,y,\dot x, \dot y) =
x^2 + y^2 + 2\frac{1-\mu}{\sqrt{(x+\mu)^2 +y^2}} + 2\frac{\mu}{\sqrt{(x-1+\mu)^2 +y^2}} -(\dot x^2 + \dot y^2),
\]
known as the Jacobi integral. In our system of coordinates, this is
\[
E(v^1, v^2, v^3, v^4,v^5,v^6) := 
(v^1)^2 + (v^3)^2 +2(1-\mu)v^5 +2\mu v^6 -((v^2)^2+(v^4)^2).
\]
Since $E$ is constant along any orbit of the system, we evaluate it at the endpoint of the first piece of the Chebyshev decomposition. This leads to a new phase condition that can replace the one previously exhibited at equation \eqref{eq:PoincareCondition}. In terms of the Chebyshev coefficients it is given by
\begin{align}
f^{(0)}(A_0) = &\left(a_{0,0}^{(1,1)} + 2 \displaystyle \sum_{k=1}^\infty a_{0,k}^{(1,1)}\right)^2 +\left(a_{0,0}^{(1,3)} + 2 \displaystyle \sum_{k=1}^\infty a_{0,k}^{(1,3)}\right)^2 +2(1-\mu)\left(a_{0,0}^{(1,5)} + 2 \displaystyle \sum_{k=1}^\infty a_{0,k}^{(1,5)}\right) \nonumber \\
 &+2\mu\left(a_{0,0}^{(1,6)} + 2 \displaystyle \sum_{k=1}^\infty a_{0,k}^{(1,6)}\right) -\left(a_{0,0}^{(1,2)} +2\displaystyle \sum_{k=1}^\infty a_{0,k}^{(1,2)}\right)^2 -\left(a_{0,0}^{(1,4)} +2\displaystyle \sum_{k=1}^\infty a_{0,k}^{(1,4)}\right)^2. \label{eq:EnergyCondition}
\end{align}

\noindent We now focus on expanding the operator we need to solve in order to obtain the coefficients of the Chebyshev expansion of the $i-$th component of the solution. In the case $k\geq 1$, each case is given by
\small
\begin{align*}
f^{(i,1)}_{\alpha,k}(A_\alpha)= 2ka_{\alpha,k}^{(i,1)} -L_i\bigg(&\alpha\lambda a_{\alpha,k\pm 1}^{(i,1)} -a_{\alpha,k\pm1}^{(i,2)}\bigg), \\
f^{(i,2)}_{\alpha,k}(A_\alpha)= 2ka_{\alpha,k}^{(i,2)} -L_i\bigg(&\alpha\lambda a_{\alpha,k\pm 1}^{(i,2)} -a_{\alpha,k\pm1}^{(i,1)} -2a_{\alpha,k\pm1}^{(i,4)} +(1-\mu)\cC(1,5,5,5) +\mu\cC(1,6,6,6) \\
&+(\mu -\mu^2)\cC(5,5,5) +(\mu^2-\mu)\cC(6,6,6)\bigg), \\
f^{(i,3)}_{\alpha,k}(A_\alpha)= 2ka_{\alpha,k}^{(i,3)} -L_i\bigg(&\alpha\lambda a_{\alpha,k\pm 1}^{(i,3)} -a_{\alpha,k\pm1}^{(i,4)} \bigg), \\
f^{(i,4)}_{\alpha,k}(A_\alpha)= 2ka_{\alpha,k}^{(i,4)} -L_i\bigg(&\alpha\lambda a_{\alpha,k\pm 1}^{(i,4)} +2a_{\alpha,k\pm1}^{(i,2)} -a_{\alpha,k\pm1}^{(i,3)} +(1-\mu)\cC(3,5,5,5) +\mu\cC(3,6,6,6)\bigg), \\
f^{(i,5)}_{\alpha,k}(A_\alpha)= 2ka_{\alpha,k}^{(i,5)} -L_i\bigg(&\alpha\lambda a_{1,k\pm 1}^{(i,5)} +\cC(1,2,5,5,5) +\mu\cC(2,5,5,5) +\cC(3,4,5,5,5)\bigg), \\
f^{(i,6)}_{\alpha,k}(A_\alpha)= 2ka_{\alpha,k}^{(i,6)} -L_i\bigg(&\alpha\lambda a_{1,k\pm 1}^{(i,6)} +\cC(1,2,6,6,6) +(\mu-1)\cC(2,6,6,6) +\cC(3,4,6,6,6) \bigg).
\end{align*}
\normalsize

\noindent The cases $k=0$ are as given in \eqref{eq:ChebyProblem} for every $i=2,\hdots,D$. In the case $i=1$, for the periodic orbit we use the condition previously given to rewrite the problem as a zero finding of an operator. Those are given by
\begin{align*}
f^{(1,1)}_{0,0}(A_\alpha) &=a_{0,0}^{(D,2)} + 2\displaystyle \sum_{k=1}^\infty a_{0,k}^{(D,2)}(-1)^k, \\
f^{(1,2)}_{0,0}(A_\alpha) &=a_{0,0}^{(D,3)} + 2\displaystyle \sum_{k=1}^\infty a_{0,k}^{(D,3)}(-1)^k, \\
f^{(1,3)}_{0,0}(A_\alpha) &=a_{0,0}^{(1,2)} + 2\displaystyle \sum_{k=1}^\infty a_{0,k}^{(1,2)}, \\
f^{(1,4)}_{0,0}(A_\alpha) &=a_{0,0}^{(1,3)} + 2\displaystyle \sum_{k=1}^\infty a_{0,k}^{(1,3)}, \\
f^{(1,5)}_{0,0}(A_\alpha) &=\left(a_{0,0}^{(1,5)} +2\displaystyle\sum_{k=1}^\infty a_{0,k}^{(1,5)}(-1)^k\right)\left|a_{0,0}^{(1,1)} +2\displaystyle\sum_{k=1}^\infty a_{0,k}^{(1,1)}(-1)^k +\mu \right| -1, \\
f^{(1,6)}_{0,0}(A_\alpha) &=\left(a_{0,0}^{(1,6)} +2\displaystyle\sum_{k=1}^\infty a_{0,k}^{(1,5)}(-1)^k\right)\left|a_{0,0}^{(1,1)} +2\displaystyle\sum_{k=1}^\infty a_{0,k}^{(1,1)}(-1)^k -1 +\mu \right| -1. 
\end{align*}
The operator is now completely determined. In order to approximate the solution we truncate 
the unknowns and the operator a similar way as for the Lorenz system and obtain an operator 
such that $F:\R^{6mD +1} \to \R^{6mD +1}$ in the case $\alpha=0,1$ and such that 
$F:\R^{6mD} \to \R^{6mD}$ in the higher dimensional cases. Again, only the search for the periodic orbit requires a good initial guess to obtain the approximation. Following \cite{Ransford} we get an initial guess on which we applied Newton's method. To present results, we regrouped different orbits with the same level of energy.

Figure \ref{fig:Energy317} illustrates two Lyapunov periodic orbits associated with the libration
points $\mathcal{L}_1$ and $\mathcal{L}_2$. Both orbit have energy $E=3.17$ and have been 
computed using $m=50$ Chebyshev nodes and $D=8$
Chebyshev domains. The periods are $L\approx 1.7122$ for the orbit around $\mathcal{L}_2$ 
and $L\approx 1.4242$ for the orbit around $\mathcal{L}_1$. 
See Figure \ref{fig:Manifold317} for an illustration of the
numerically computed parameterized invariant manifolds.

\begin{figure}
 \begin{center}
  \includegraphics[width=\textwidth]{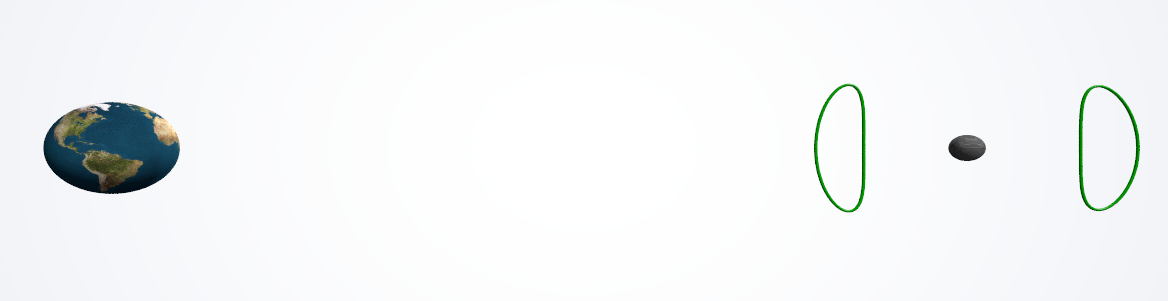}
  \caption{A pair of Lyapunov orbits around the libration points $\mathcal{L}_1$ and $\mathcal{L}_2$. 
  The computation displayed were done in the case of the earth-moon ratio $\mu=0.0123$. Both orbit have energy $3.17$.
  The sizes of the Earth and Moon are not to scale, and the orientation of the Earth chosen for effect rather than accuracy
  (the actual orientation would have us looking down at the North Pole of the Earth).} \label{fig:Energy317}
 \end{center} 
\end{figure}

\begin{figure}
 \begin{center}
  \includegraphics[width=\textwidth]{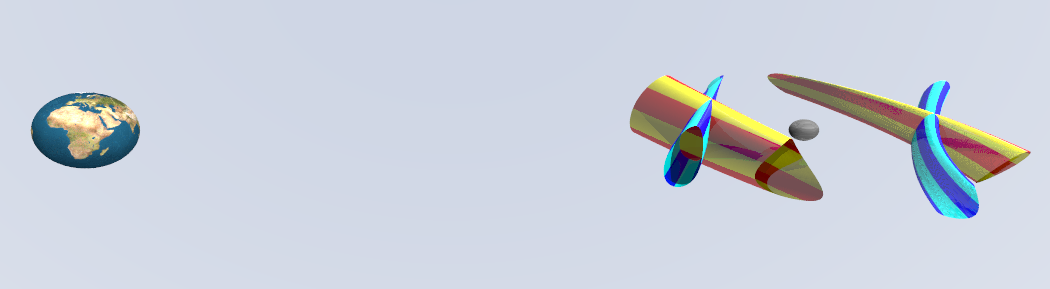} 
  \caption{Stable (light and dark blue) and unstable (red and yellow) manifold for both the Lyapunov 
  orbits already discussed in Figure \ref{fig:Energy317}. Manifolds are computed with $N=10$ Taylor nodes, $m=50$ Chebyshev coefficients
  per $D=8$ Chebyshev domains. The component displayed are $v^1,v^2,v^3$, 
  that is $x,\dot x$ and $y$ in the original system of coordinates (again the sizes of the primaries are not to scale and orientation is 
  chosen for effect).} \label{fig:Manifold317}
 \end{center}
\end{figure}

We must stress that these computations were preformed for the polynomial 
equation $\dot v = g(v(t))$, where $g$ is given by \eqref{eq:g_CRTBP}. To show
that a point laying in an invariant manifolds in the polynomial problem 
corresponds to a point on the manifold for the original CRTBP we use the following Theorem.  
Intuitively this works because of the conjugacy described in 
Equation \eqref{eq:autoDiff}.

\begin{theorem}
Let $P(t,\sigma)$ be a parameterization of the local stable manifold of a given periodic orbit $v(t)$ 
satisfying $\dot v = g(v(t))$, where $g$ is given by \eqref{eq:g_CRTBP}. If $x_0 \in \R^6$ is such 
that $x_0=P(t_0,\sigma_0)$, then the point $y_0 \in \R^4$ given by the first four component of 
$x_0$ is in the stable set of the corresponding periodic orbit of $\dot u = f(u)$, where $f$ is given 
by \eqref{eq:CRTBP_non_polynomial}. 
\end{theorem}

\begin{proof}
Note that using the same remark as in the proof of Lemma \eqref{Lemma:periodicCRTBP}, 
we obtain from the periodic orbit $v(t)$ a periodic orbit of the CRTBP. That is a periodic solution 
of $\dot u = f(u(t))$. We denote this orbit  by $u(t)$ and we have that $v^i(t)=u^i(t)$ for $i=1,2,3,4$, while
\[
v^5(t)= \frac{1}{\sqrt{(u^1(t)+\mu)^2 + u^3(t)^2}}~\mbox{and}~v^6(t)= \frac{1}{\sqrt{(u^1(t)-1+\mu)^2 + u^3(t)^2}}.
\]
Let $x(t)$ denote the trajectory obtained by flowing $x_0$ forward in time by $\Phi(v,t)$ the flow solution of $\dot v = g(v(t))$. 
So that, using the conjugacy relation \eqref{eq:FlowConjugacy}, we have
\[
x(t)= P(t_0+t,e^{\lambda t}\sigma_0),
\]
where $\lambda$ is the stable eigenvalue associated to the orbit. Moreover, it follows that
\begin{align*}
\dot x(t) &= \frac{\partial}{\partial t} P(t_0+t,e^{\lambda t}\sigma_0) \\
	  &= \frac{\partial}{\partial t} \Phi(x_0,t) \\
	  &= g(x(t)). \\
\end{align*}
Thus, by definition of $g$  
\[
x^5(t)= \frac{1}{\sqrt{(x^1(t)+\mu)^2 +(x^3(t))^2}} +C_1
\]
and
\[
x^6(t)= \frac{1}{\sqrt{(x^1(t)-1+\mu)^2 +(x^3(t))^2}} +C_2.
\]
Here $C_1,C_2$ are arbitrary constants. But $x_0 \in W^s(v)$, so that for any $\epsilon>0$ there exists $T$ such that for all $t\geq T$, we have that
\[
\min_{s\in [0,\omega]} \left| x^5(t) - v^5(s) \right| < \epsilon.
\]
This force $C_1=0$. Similarly, we have that $C_2=0$. Now $x^5$ and $x^6$ in $g$ are rewritten with the first four component so that $g$ reduce to $f$ and $y(t)=(x^1(t),x^2(t),x^3(t),x^4(t))$ satisfies $\dot y(t) = f(y(t))$. Moreover, since $x_0 \in W^s(v)$, we have that $y(t)$ is in the stable set of the corresponding orbit $u(t)$, as desired.
\end{proof}

\subsubsection{Connecting orbits as solutions of boundary value problems}
As in Section \ref{sec:shortConnection} we now use the parametrized manifolds to compute 
connecting orbits between period orbits. 
In the present section we consider heteroclinic orbits for the 
CRTBP and we do not find any short connections.
Instead, we solve a two point boundary value problem with
 the manifolds as boundary conditions.  
 This strategy is standard and is for example discussed in detail in
 \cite{MR1007358,MR1205453,MR2454068,MR2511084}.
 
 The references just cited obtain boundary conditions by 
 projecting onto the linear approximation of the stable/unstable manifolds
 given by the stable/unstable normal bundles associated with the periodic orbit.
 Projecting instead onto high order parameterizations
 can substantially reduce the integration time and 
numerically stabilize the problem.
The behavior of the connecting orbit on the manifold is then recovered 
via the flow conjugacy.
\begin{remark}
In this work we compute connecting orbits by numerically integrating the system. However, one could adapt the approach developed in section \ref{Sec:ChebyshevExpansion} with $\alpha=0$ to compute any orbit solution of a given boundary value problem expanded as Chebyshev series. This has been done in 
\cite{LessardReinhardt, MR3353132, RayJB}
and even leads to computer assisted proofs.  We return to this remark in an upcoming work. 
We also refer to 
\cite{MR1947690, MR2112702, MR3032848, WilczakZgli}
for further reading about 
computer assisted proofs in the CRTBP.
\end{remark}

We let $P$ and $Q$ denote 
parameterizations of the stable and unstable manifolds, and 
seek $(\theta_u,\sigma_u)$, $(\theta_s,\sigma_s)$ 
and an integrating time $T$ -- or ``time of flight'' -- such that
\begin{align}\label{eq:ConnectingOrbitProblem}
\Phi(P_0,T)&= P(\theta_s,\sigma_s),
\end{align}
where $P_0= Q(\theta_u,\sigma_u)$. Equation \eqref{eq:ConnectingOrbitProblem} 
has five unknowns, namely the integrating time and the parameters on both manifold. 
Since orbits of the CRTBP lie in $\R^4$, Equation \eqref{eq:ConnectingOrbitProblem} 
provides only four equations and we have more unknowns than equations.
As a result we cannot expect to isolate a solution.
To remedy the situation we simply fix $\sigma_u=-1$, removing one of the variables. 
This corresponds to a choice of boundary components for the local unstable manifold.
\begin{remark}
 Recall that the energy is constant along solution curves, and it's impossible to find a connecting orbit between two periodic solution with different energy level. We avoided this problem by introducing the energy as the phase condition when we solve for the underlying orbit. Recall the definition of $f^{(0)}$ in \eqref{eq:EnergyCondition}.
\end{remark}

Using Newton's method with the unstable manifold of the orbit on the right in figure \ref{fig:Energy317} and the stable manifold of the orbit on the left we found an approximation of a solution to this problem where
\begin{align*}
(\theta_u,\sigma_u) &\approx (3.086681925168687,-1) \\
(\theta_s,\sigma_s) &\approx (0.065696587097979, 1) \\
 T &\approx 1.654424821513812.
\end{align*}

Both manifolds were computed with $m=50$, $D=8$, $N=50$, $K=5$ and $k_0=10$. To find an initial guess on which
 to apply Newton's method we integrated $40$ points evenly distributed on the boundary of the unstable manifold and 
 observed that some orbits were potentially intersecting the stable manifold. The connecting orbit and the two manifold 
 are displayed in figure \ref{fig:ConnectingCRTBP}.

\begin{figure}
    \centering
    {\includegraphics[width=.5\textwidth]{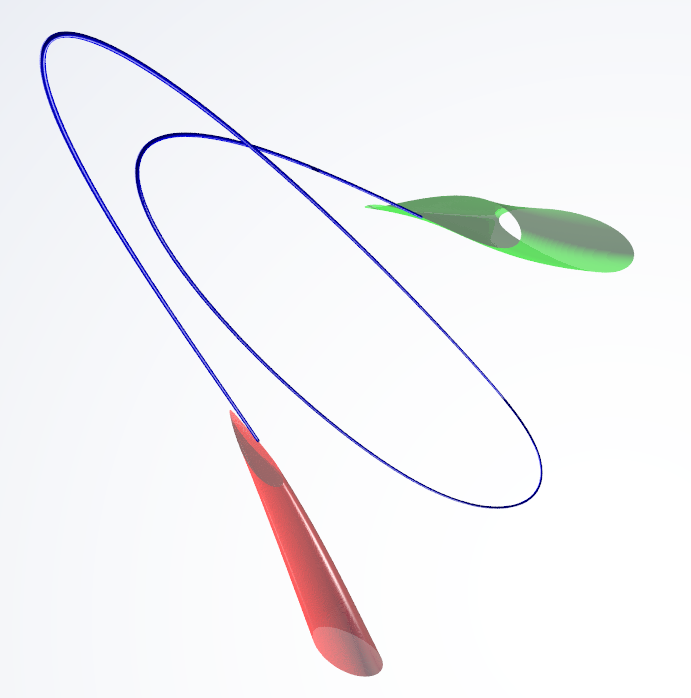} }
    \caption{Connecting orbit between the two 
    Lyapunov orbits in the CRTBP. (The Moon -- not shown -- would be in the middle of the ``loop'').  
    We use both parameterized local manifolds to reduce the problem 
    to a finite time interval.  The resulting finite time interval is much shorter using these high order 
    parameterizations than it would be if we only projected onto the linear approximation. 
    This shortening of the ``time of flight'' has the effect of stabilizing the numerical solution of the BVP.
    The stable manifold is in green while the unstable is in red. The blue curve is
    the heteroclinic connection, which we find by solving a boundary value problem with end points on 
    the parameterizations.}\label{fig:ConnectingCRTBP}
\end{figure}

We also use the conjugacy relation to extend the connecting orbit forward and backward 
on the manifolds.  Integrating in the parameter space until
$\sigma_s \approx 10^{-15}$ takes  
\[
t \approx 12.3688,
\]
and for the backward trajectory 
\[
t \approx 16.1050.
\]
The full trajectory is displayed in figure \ref{fig:ExtendedConnecting}. Note that out of the three pieces of the trajectory, 
only the one in blue was obtained by numerically integrating the system and this piece required 
a time step of only $2$ time units.  Redoing the computation but projecting onto the 
linear approximations would result in a time of flight of roughly $30$ time units.

\begin{figure}
    \centering
    \subfigure{\includegraphics[width=.32\textwidth]{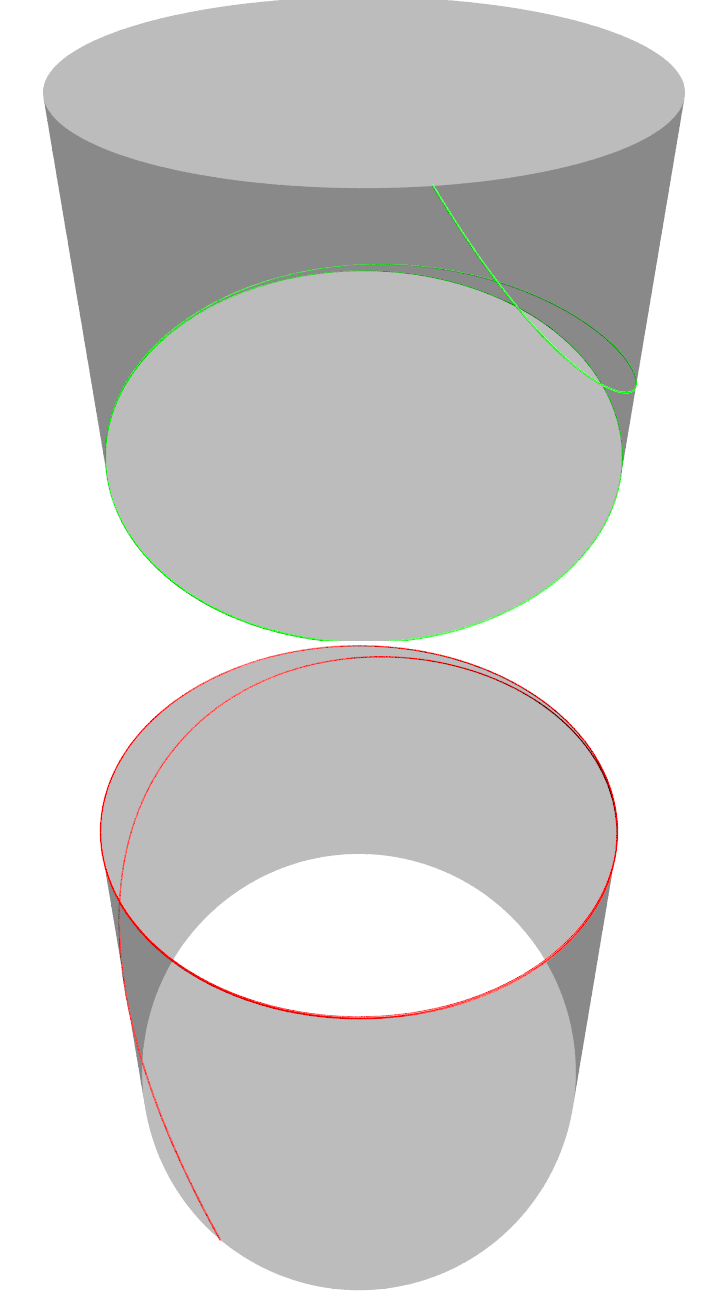} }
    \subfigure{\includegraphics[width=.64\textwidth]{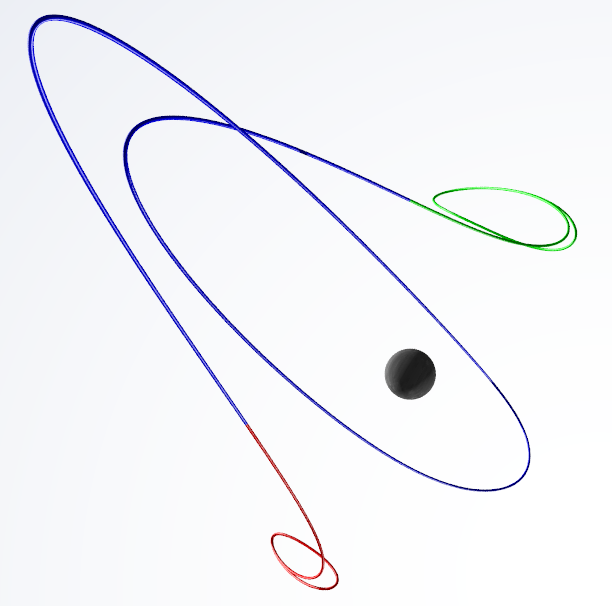} }
    \caption{Extension of the connecting orbit using the parameterization of the manifolds and 
    the conjugacy to linear dynamics. The trajectories in red and green are obtained using the 
    linear dynamics in parameter space rather than integrating in the phase space -- hence there is 
    no threat that numerical instabilities ``kick us off'' the manifold. On the left, the trajectories are displayed in the parameters
    space. The top half of the parameter space is displayed for the stable case since $\sigma$ is positive at all time and the 
    bottom half is displayed for the unstable case. The right frame shows the same orbits lifted
    to the phase space by the parameterization.  Compare with the schematic 
    given in Figure \ref{fig:conjSketch}.}\label{fig:ExtendedConnecting}
\end{figure}

\subsection{A Circular Restricted Four Body Problem}
We now consider a gravitational problem consisting of
three massive bodies (again called the primaries)
located at the vertices of an equilateral 
triangle in the central configuration of Lagrange.  These bodies rotate 
in circular orbits about their common center of mass, all with the same period, 
rigidly fixing the triangular formation.  After changing to a
co-rotating frame we are interested the motion of a 
massless fourth particle moving in the gravitational field of the 
primaries.  In the present work we suppose that the massless particle
moves in the plane of the primaries.  


\begin{figure}
    \centering
    \subfigure{{\includegraphics[width=.7\textwidth]{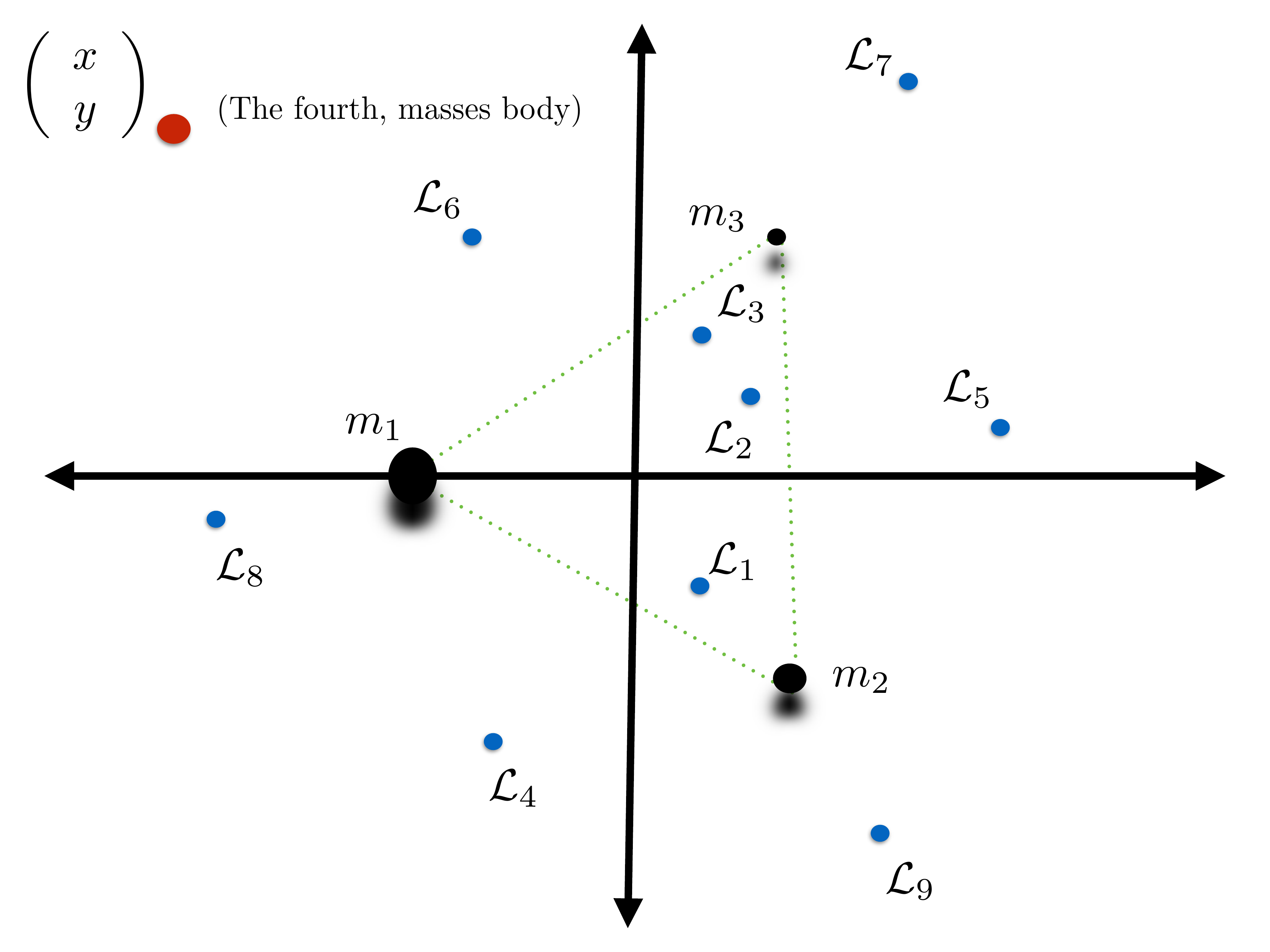} }}
\caption{Schematic representation of the CRFBP: 
three massive bodies -- the primaries -- are arranged in an equilateral
triangle configuration.  After changing to co-rotating coordinates which fix the location 
of the massive bodies, we study a fourth and massless particle moving in the 
gradational field of the primaries.  In the present work we restrict
the fourth body to move in the plane defined by the equilateral triangle.
The figure illustrates the approximate locations of nine librations points (relative equilibria)
occurring for the mass ratios considered in the present work.  The libration points are 
denoted $\mathcal{L}_j$ for $1 \leq j \leq 9$.   
Below, we study invariant manifolds
attached to a hyperbolic Lyapunov orbit about $\mathcal{L}_8$.}\label{fig:fourBodies}
\end{figure}

It is standard practice to 
normalize the masses of the primaries so that $0 < m_3 \leq m_2 \leq m_1$,
and 
\[
m_1 + m_2 + m_3 = 1.
\] 
The rotating coordinates are chosen so that 
 the center of mass is at the origin, the largest
 primary is on the $x$-axis, the $x$-axis 
cuts the side of the triangle opposite the largest primary,
and the smallest primary is in the first quadrant.  
More explicitly,  the primaries are located at positions 
\[
p_1 = (x_1, y_1,z_1), \quad \quad 
p_2 = (x_2, y_2,z_2), \quad \quad
\text{and} \quad \quad
p_3 = (x_3, y_3,z_3),
\]
with 
\begin{eqnarray*}
x_1 &=&   \frac{-|K| \sqrt{m_2^2 + m_2 m_3 + m_3^2}}{K} \\
y_1 &=&   0 \\
z_1 &=& 0
\end{eqnarray*}
\begin{eqnarray*}
x_2 &=&  \frac{|K|\left[(m_2 - m_3) m_3 + m_1 (2 m_2 + m_3)  \right]}{
2 K \sqrt{m_2^2 + m_2 m_3 + m_3^2}
}  \\
y_2 &=&  \frac{-\sqrt{3} m_3}{2 m_2^{3/2}} \sqrt{\frac{m_2^3}{m_2^2 + m_2 m_3 + m_3^2}} \\
z_2 &=& 0
\end{eqnarray*}
and
\begin{eqnarray*}
x_3 &=&  \frac{|K|}{2 \sqrt{m_2^2 + m_2 m_3 + m_3^2}}  \\
y_3 &=&  \frac{\sqrt{3}}{2 \sqrt{m_2}} \sqrt{\frac{m_2^3}{m_2^2 + m_2 m_3 + m_3^2}} \\
z_3 &=& 0
\end{eqnarray*}
where 
\[
K = m_2(m_3 - m_2) + m_1(m_2 + 2 m_3).
\]
Define the potential function 
\[
\Omega(x,y,z) :=
\frac{1}{2} (x^2 + y^2) + \frac{m_1}{r_1(x,y,z)} + \frac{m_2}{r_2(x,y,z)} + \frac{m_3}{r_3(x,y,z)}, 
\]
with
\[
r_1(x,y,z) := \sqrt{(x-x_1)^2 + (y-y_1)^2 +(z-z_1)^2}, 
\]
\[
r_2(x,y,z) := \sqrt{(x-x_2)^2 + (y-y_2)^2 +(z-z_2)^2},
\]
and
\[
r_3(x,y,z) := \sqrt{(x-x_3)^2 + (y-y_3)^2 +(z-z_3)^2}.
\]
The equations of motion for the massless particle 
in the co-rotating coordinates are 
\begin{equation}\begin{split}\label{ecuacionesfinales}
\ddot{x}-2\dot{y}&=\Omega_{x},\\
\ddot{y}+2\dot{x}&=\Omega_{y},\\
\ddot{z} &= \Omega_z,
\end{split}
\end{equation}
where
\[
\frac{\partial}{\partial x} \Omega = 
\Omega_x(x,y,z) = x - \frac{m_1(x - x_1)}{r_1(x,y,z)^3}
- \frac{m_2(x - x_2)}{r_2(x,y,z)^3} - \frac{m_3(x - x_3)}{r_3(x,y,z)^3},
\]
\[
\frac{\partial}{\partial y} \Omega = 
\Omega_y(x,y,z) = y - \frac{m_1(y - y_1)}{r_1(x,y,z)^3}
- \frac{m_2(y - y_2)}{r_2(x,y,z)^3} - \frac{m_3(y - y_3)}{r_3(x,y,z)^3},
\]
and
\[
\frac{\partial}{\partial z} \Omega = 
\Omega_z(x,y,z) = -\frac{m_1(z - z_1)}{r_1(x,y,z)^3}
- \frac{m_2(z - z_2)}{r_2(x,y,z)^3} - \frac{m_3(z - z_3)}{r_3(x,y,z)^3}.
\]

The problem is the subject of many studies beginning with the work of 
\cite{MR510556}.  Detailed analysis of the equilibrium solutions and their 
stability are found in \cite{MR3176322, MR2232439, MR2845212, MR3463046, MR2027748},
while periodic orbits are studied in 
\cite{MR3571218, BaltagiannisPapadakis, MR3500916, burgosPeriodicOrbits, jpHalo}.
The earlier study of \cite{MR2917610} considers stable/unstable manifolds attached
to periodic orbits (using the linear approximation given by 
the stable/unstable normal bundles combined with numerical integration).
More complex dynamical behavior such as heteroclinic/homoclinic phenomena 
and transport are studied numerically in \cite{MR3304062, MR2596303, MR2013214, shaneJay}.
A Hill's approximation is derived in \cite{MR3346723}.
See also the more theoretical 
studies of \cite{MR3326646, MR3626383, MR3158025}.

The CRFBP has $8, 9$, or $10$ equilibrium points depending on the values of the 
mass parameters.  These will have either center $\times$ center, center $\times$ saddle,
or saddle-focus stability, depending on the ratios of the masses.  
As in the CRTBP, the center $\times$ saddle equilibria give rise to one parameter
families of hyperbolic Lyapunov orbits.  We compute local stable/unstable 
manifolds attached to one of these below.

Using automatic differentiation 
we derive a related polynomial vector field
\small
\begin{align*}
g(v)= \begin{pmatrix}
v^2 \\
2v^4 +v^1 -m_1v^1(v^7)^3 -m_2 v^1(v^8)^3 -m_3v^1(v^9)^3 +m_1x_1(v^7)^3 +m_2x_2(v^8)^3 +m_3x_3(v^9)^3 +\beta v^2 \\
v^4 \\
-2v^2 +v^3 -m_1v^3(v^7)^3 -m_2 v^3(v^8)^3 -m_3v^3(v^9)^3 +m_1y_1(v^7)^3 +m_2y_2(v^8)^3 +m_3y_3(v^9)^3 \\
v^6 \\
-m_1v^5(v^7)^3 -m_2 v^5(v^8)^3 -m_3v^5(v^9)^3 +m_1z_1(v^7)^3 +m_2z_2(v^8)^3 +m_3z_3(v^9)^3 \\
-v^1v^2(v^7)^3 -v^3v^4(v^7)^3 -v^5v^6(v^7)^3 +x_1v^2(v^7)^3 +y_1v^4(v^7)^3 +z_1v^6(v^7)^3 +\alpha_1 (v^7)^3 \\
-v^1v^2(v^8)^3 -v^3v^4(v^8)^3 -v^5v^6(v^8)^3 +x_2v^2(v^8)^3 +y_2v^4(v^8)^3 +z_2v^6(v^8)^3 +\alpha_2 (v^8)^3 \\
-v^1v^2(v^9)^3 -v^3v^4(v^9)^3 -v^5v^6(v^9)^3 +x_3v^2(v^9)^3 +y_3v^4(v^9)^3 +z_3v^6(v^9)^3 +\alpha_3 (v^9)^3
\end{pmatrix}.
\end{align*}
\normalsize
The mapping $R$ used in the automatic differentiation is defined  
\cite{shaneJay}, but is similar to the mapping discussed above for the 
three body case.

The constants terms $\beta,\alpha_1,\alpha_2,\alpha_3$
are Lagrange multipliers, which are needed to isolate a periodic solution 
(this time we will not impose any symmetry, hence the boundary condition constraints remain 
and have to be balanced).
The following result, whose proof is found in \cite{jpHalo}, explains the relation between the
polynomial and non-polynomial problems.

\begin{lemma} \label{lemmaFourBody}
Assume that $\beta,\alpha_1,\alpha_2,\alpha_3, L \in \R$ are fixed constant with $L>0$ and let $\textbf{n},\textbf{p} \in \R^6$ be fixed vector. Suppose that $u:[0,2L] \to \R^9$ is a periodic solution of $\dot v(t) = g(v)$ with $g$ as above and
\begin{align*}
0 &= \textbf{n} \cdot ( (u^1(0),u^2(0),u^3(0),u^4(0),u^5(0),u^6(0))^T-\textbf{p}) \\
u^7(0) &= \frac{1}{\sqrt{(u^1(0)-x_1)^2 + (u^3(0)-y_1)^2 +(u^5(0)-z_1)^2}} \\
u^8(0) &= \frac{1}{\sqrt{(u^1(0)-x_2)^2 + (u^3(0)-y_2)^2 +(u^5(0)-z_2)^2}} \\
u^9(0) &= \frac{1}{\sqrt{(u^1(0)-x_3)^2 + (u^3(0)-y_3)^2 +(u^5(0)-z_3)^2}}
\end{align*}
and that $u^7(t),u^8(t),u^9(t)>0$ for all $t \in [0,2L]$. Then
\begin{enumerate}
 \item $\beta=\alpha_1=\alpha_2=\alpha_3=0$
 \item the function $\hat u : [0,2L] \to \R^6$ given by
 \[
 \hat u(t)= (u^1(0),u^2(0),u^3(0),u^4(0),u^5(0),u^6(0))^T
 \]
 is a periodic solution of the four body problem.
\end{enumerate}
\end{lemma}

Yet, the variables $\alpha_1,\alpha_2,\alpha_3$ are not necessary
 in the Chebyshev setting since one can use the following lemma to force the initial value condition on $v^7,v^8,v^9$ without introducing additional equations. The proof is omitted since it is similar to the proof of Lemma \ref{Lemma:periodicCRTBP}.
\begin{lemma}\label{lemma:4body}
Let $v^1(t),v^2(t),v^3(t),v^4(t),v^5(t),v^6(t)$ be periodic solution with 
same period $\omega$ and such that 
$\dot v^1 = v^2$, $\dot v^3 = v^4$, $\dot v^5 = v^6$. Let $v^7(t),v^8(t),v^9(t)$ satisfy
\begin{align*}
\dot v^7(t) &= -(v^7(t))^3\left((v^1(t)-x_1)v^2(t) +(v^3(t)-y_1)v^4(t) +(v^5(t)-z_1)v^6(t)\right) \\
v^7(0) &= \frac{1}{\sqrt{(v^1(0)-x_1)^2 + (v^3(0)-y_1)^2 +(v^5(0)-z_1)^2}} \\
\dot v^8(t) &= -(v^8(t))^3\left((v^1(t)-x_2)v^2(t) +(v^3(t)-y_2)v^4(t) +(v^5(t)-z_2)v^6(t)\right) \\
v^8(0) &= \frac{1}{\sqrt{(v^1(0)-x_2)^2 + (v^3(0)-y_2)^2 +(v^5(0)-z_2)^2}} \\
\dot v^9(t) &= -(v^9(t))^3\left((v^1(t)-x_3)v^2(t) +(v^3(t)-y_3)v^4(t) +(v^5(t)-z_3)v^6(t)\right)  \\
v^9(0) &= \frac{1}{\sqrt{(v^1(0)-x_3)^2 + (v^3(0)-y_3)^2 +(v^5(0)-z_3)^2}}.
\end{align*}
Then $v^7(t),v^8(t)$ and $v^9(t)$ are periodic with period $\omega$.
\end{lemma}

The extra condition balancing the system is the Poincar\'{e} condition which rewrites exactly as in \eqref{eq:PoincareCondition}. This condition rejects potential time translation of a periodic solution. The other conditions are coming from automatic differentiation and are given by
\begin{align*}
v^7(0)^2\big( (v^1(0)-x_1)^2 +(v^3(0)-y_1)^2 +(v^5(0)-z_1)^2 \big) -1 =0, \\
v^8(0)^2\big( (v^1(0)-x_2)^2 +(v^3(0)-y_2)^2 +(v^5(0)-z_2)^2 \big) -1 =0, \\
v^9(0)^2\big( (v^1(0)-x_3)^2 +(v^3(0)-y_3)^2 +(v^5(0)-z_3)^2 \big) -1 =0.
\end{align*}

The boundary condition for each Chebyshev subdomain being used for this problem define the operators $f^{(i,j)}_{\alpha,0}$ as previously given in \eqref{eq:ChebyProblem} and \eqref{eq:periodicity}. For all $1\leq i \leq D$, the case for $k\geq 1$ are given by
\small
\begin{align*}
f^{(i,1)}_{\alpha,k}(A) = 2k a_{\alpha,k}^{(i,1)} +L_i \bigg( &-\alpha\lambda a_{\alpha,k\pm1}^{(i,1)}  +a_{\alpha,k\pm 1}^{(i,2)} \bigg), \\
f^{(i,2)}_{\alpha,k}(A) = 2k a_{\alpha,k}^{(i,2)} +L_i \bigg( &-\alpha\lambda a_{\alpha,k\pm1}^{(i,2)}  +2a_{\alpha,k\pm 1}^{(i,4)} +a_{\alpha,k\pm 1}^{(i,1)} +\beta a_{\alpha,k\pm 1}^{(i,2)}\\
                                          &-m_1\cC(1,7,7,7) -m_2\cC(1,8,8,8) -m_3\cC(1,9,9,9) \\
                                          &+m_1x_1\cC(7,7,7) +m_2x_2\cC(8,8,8) +m_3x_3\cC(9,9,9)\bigg), 
                                          \end{align*}
                                          \begin{align*}
f^{(i,3)}_{\alpha,k}(A) = 2k a_{\alpha,k}^{(i,3)} +L_i \bigg( &-\alpha\lambda a_{\alpha,k\pm1}^{(i,3)}  +a_{\alpha,k\pm 1}^{(i,4)} \bigg), \\
f^{(i,4)}_{\alpha,k}(A) = 2k a_{\alpha,k}^{(i,4)} +L_i \bigg( &-\alpha\lambda a_{\alpha,k\pm1}^{(i,4)} -2a_{\alpha,k\pm 1}^{(i,2)} +a_{\alpha,k\pm 1}^{(i,3)} \\
                                          &-m_1\cC(3,7,7,7) -m_2\cC(3,8,8,8) -m_3\cC(3,9,9,9) \\
                                          &+m_1y_1\cC(7,7,7) +m_2y_2\cC(8,8,8) +m_3y_3\cC(9,9,9)\bigg), 
                                          \end{align*}
                                          \begin{align*}
f^{(i,5)}_{\alpha,k}(A) = 2k a_{\alpha,k}^{(i,5)} +L_i \bigg( &-\alpha\lambda a_{\alpha,k\pm1}^{(i,5)}  +a_{\alpha,k\pm 1}^{(i,6)} \bigg), \\
f^{(i,6)}_{\alpha,k}(A) = 2k a_{\alpha,k}^{(i,6)} +L_i\bigg( &-\alpha\lambda a_{\alpha,k\pm1}^{(i,6)} -m_1\cC(5,7,7,7) -m_2\cC(5,8,8,8) -m_3\cC(5,9,9,9) \\
                                          &+m_1z_1\cC(7,7,7) +m_2z_2\cC(8,8,8) +m_3z_3\cC(9,9,9)\bigg), 
                                          \end{align*}
\begin{align*}
f^{(i,7)}_{\alpha,k}(A) = 2k a_{\alpha,k}^{(i,7)} +L_i \bigg( &-\alpha\lambda a_{\alpha,k\pm1}^{(i,7)} -\cC(1,2,7,7,7) -\cC(3,4,7,7,7)-\cC(5,6,7,7,7) \\
                                          &+x_1\cC(2,7,7,7) +y_1\cC(4,7,7,7) +z_1\cC(6,7,7,7) \bigg), 
                                          \end{align*}
                                          \begin{align*}
f^{(i,8)}_{\alpha,k}(A) = 2k a_{\alpha,k}^{(i,8)} +L_i \bigg( &-\alpha\lambda a_{\alpha,k\pm1}^{(i,8)} -\cC(1,2,8,8,8) -\cC(3,4,8,8,8)-\cC(5,6,8,8,8) \\
                                          &+x_2\cC(2,8,8,8) +y_2\cC(4,8,8,8) +z_2\cC(6,8,8,8) \bigg), 
                                          \end{align*}
                                          \begin{align*}
f^{(i,9)}_{\alpha,k}(A) = 2k a_{\alpha,k}^{(i,9)} +L_i \bigg( &-\alpha\lambda a_{\alpha,k\pm1}^{(i,9)} -\cC(1,2,9,9,9) -\cC(3,4,9,9,9)-\cC(5,6,9,9,9)\\
                                          &+x_3\cC(2,9,9,9) +y_3\cC(4,9,9,9) +z_3\cC(6,9,9,9) \bigg), 
\end{align*}
\normalsize
Again, we solve recursively the truncated operator to obtain an approximation of the manifold. The case $\alpha=0$ will have the extra variable $\beta$ from Lemma \eqref{lemmaFourBody} and the case $\alpha=1$ will have the eigenvalue as an extra unknown. Note that in this case we fix the frequency $L$ to a constant, it is still possible to find a solution for the same reason as mentioned in the case of the CRTBP. In figure \ref{fig:ManifoldCRFBP}, we display an unstable manifold for a planar 
Lyapunov orbit about $\mathcal{L}_8$. The computations were done with $D=4$ and $N=60$. We used masses
\[
m_1\approx 0.9987 ,~m_2 \approx0.0010 ~\mbox{and}~ m_3\approx 0.0003.
\]

We remark that, for the purposes of the present demonstration, we could have taken $m_1, m_2, m_3$
to have more or less any values.
Nevertheless the particular choice of mass values was suggested in \cite{jaimeComm}.  The values
of $m_1, m_2$ correspond respectively to the mass of the primary star in the binary system Epsilon Reticuli
(Henry Draper Catalogue number $27442$), and the mass of an extra solar planet discovered 
orbiting HD $27442$ in the year $2000$.  The mass ratio of these bodies is such that 
the system could form a CRFBP with a third trojan object. 
Our value of $m_3$ corresponds to that of a supposed Saturn like planet forming an 
equilateral triangle with $m_1$ and $m_2$.  We refer the reader
to \cite{MR2511264} for more extensive discussion of exoplanets.

We consider a Lyapunov orbit near the heaviest mass (an orbit about the 
libration point $\mathcal{L}_8$), and for this choice a uniform mesh was not suitable.
That is: it gets harder to obtain an accurate approximation as the orbit approach one of the 
heavy body. Thus, we took
\[
L_1= 0.0907L,~L_2=0.607L,~L_3 = .22L~\mbox{and}~ L_4= 0.0823L.
\]
Note that these sum to $1$, so that the integrating time for the orbit is preserved. 
To understand why the accuracy is affected as the orbit approach a body, recall that the variables arising from 
automatic differentiation are inversely proportional to the distance between the object and the corresponding 
primary, thus provoking a considerable change in the amplitude in the additional variables. \\
\begin{figure}%
    \centering
    \subfigure{{\includegraphics[width=.48\textwidth]{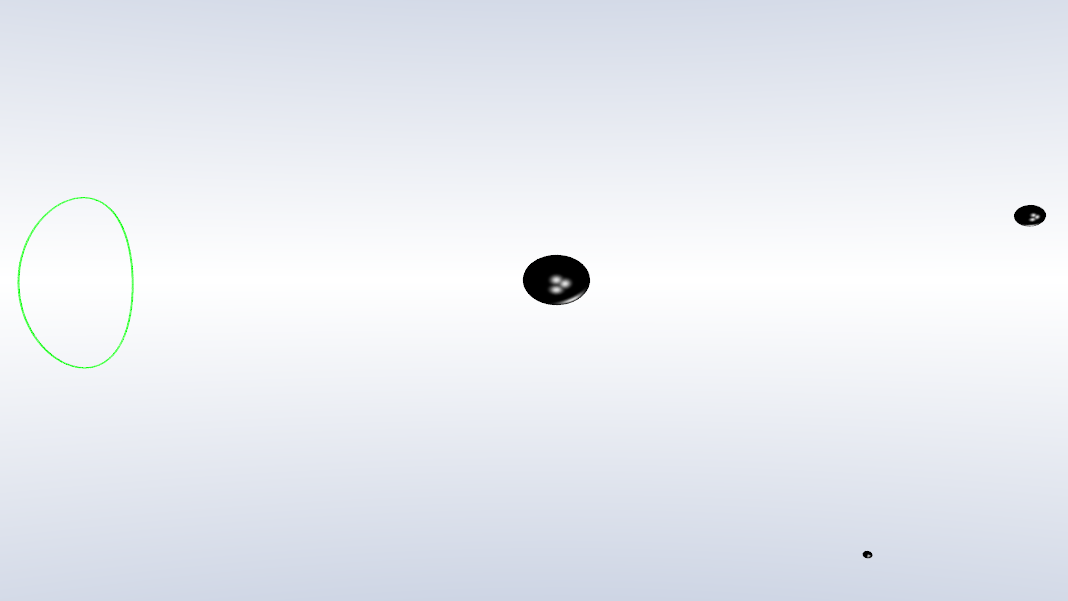} }}
    \subfigure{{\includegraphics[width=.48\textwidth]{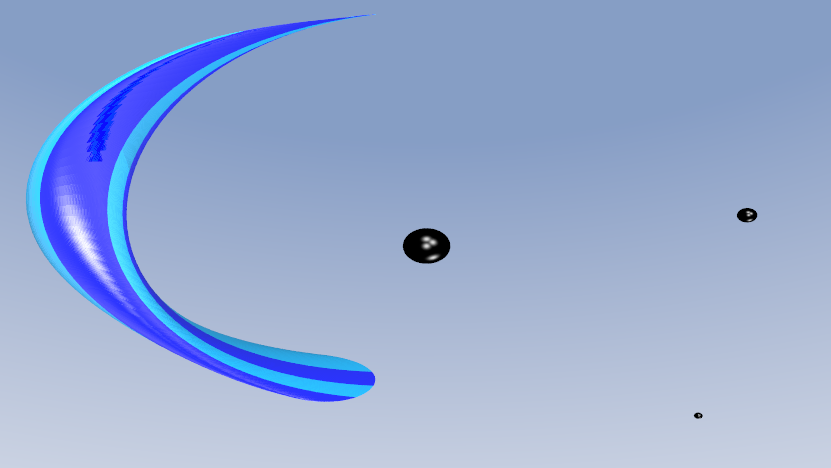} }}
    \caption{Planar Lyapunov orbit about CRFBP $\mathcal{L}_8$ (left) and its parameterized
    local stable manifold (right). The components displayed are $(x,y)$ 
    on the left and $(x,\dot x, y)$ on the right.}\label{fig:ManifoldCRFBP}
\end{figure}


\subsubsection{Connecting Orbits}
We now consider a homoclinic connection, and 
hence there are no energy considerations. 
For the planar Lyapunov periodic orbit displayed in figure \ref{fig:ManifoldCRFBP} 
we computed both the stable and unstable manifold and  apply a similar BVP approach 
as for the CRTBP.  Again, the 
connecting orbit starts on the boundary of the unstable manifold and 
ends on the boundary of the stable manifold. The conjugacy relation is used to 
compute the asymptotic behavior without phase space 
integration. In this case both eigenvalues have the same 
value with opposite signs, so the integrating time forward or backward needed is 
the same. 

We remark that to get from the boundary of the 
invariant manifold to $\sigma \approx 10^{-8}$ close to the periodic orbit
(so that the error in the linear approximation is on the order of machine epsilon)
one would need to integrate for roughly $t=342$ time units.  So this problem 
illustrates starkly the utility of using the local parameterizations
to absorb such a substantial portion of the homoclinic orbit.

Both manifolds are computed using $N=50$, $D=4$, $m=50$, $K=2$ and $k_0=10$. 
We computed one connecting orbit for each component of the local manifold 
boundaries. In the case of 
$\sigma_u=1$, the coordinates for the connecting orbit are
\begin{align*}
(\theta_u,\sigma_u) &\approx (0.502504125750113,1) \\
(\theta_s,\sigma_s) &\approx (5.287357153093578,1) \\
T &\approx 24.335325092442929.
\end{align*}
For the case of $\sigma_u=-1$, the coordinates are
\begin{align*}
(\theta_u,\sigma_u) &\approx (1.6623173901,-1) \\
(\theta_s,\sigma_s) &\approx (4.7735705589, -1) \\
T &\approx 41.3130392127.
\end{align*}
The sign of the value $\sigma$ is affected by the choice of the eigenvector, 
i.e. the sign determines the polarity of the embedding.
 In this case we picked the eigenvectors so that the boundaries have the same 
 sign when they lay on the same side of the orbit in the choice of coordinates displayed.

The reason the dynamics on the parameterized manifolds are so slow in 
this example is that the Floquet exponent $\lambda \approx \pm 0.0538$ is much closer to zero
than in any previous example. The connecting orbits are displayed in figure \ref{fig:ConnectingOrbit4B} 
along with both manifolds. The extension of the first orbit using the conjugacy relation is displayed in 
figure \ref{fig:ExtendedConnecting4B}. In both cases the coordinates displayed are $(x,y,\dot y)$.

\begin{figure}
 \centering
 \includegraphics[width=.9\textwidth]{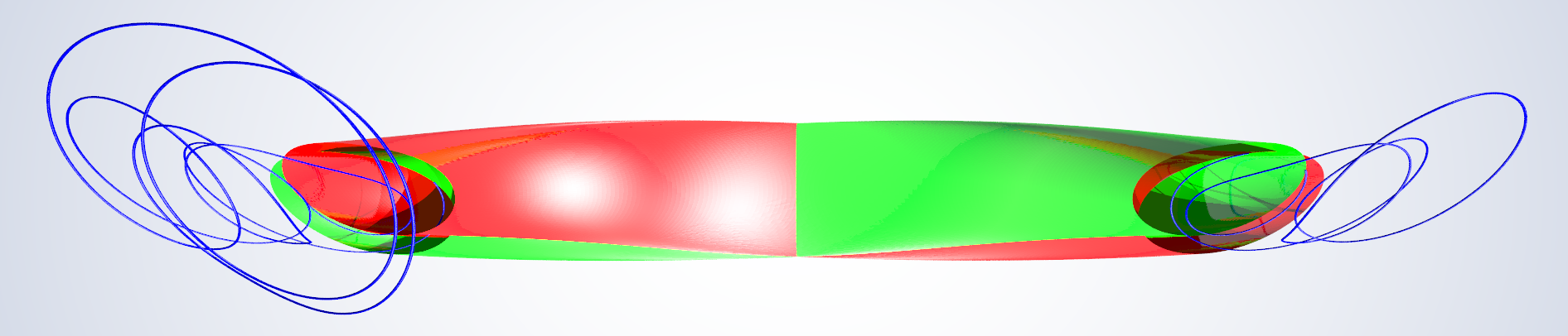}
 \caption{Homoclinic connections to the orbit displayed in figure \ref{fig:ManifoldCRFBP}.  The 
 coordinates are $x, \dot x, y$. The stable manifold is displayed in green while the unstable is in red.
 Homoclinics are the blue space curves.}\label{fig:ConnectingOrbit4B}
\end{figure}

\begin{figure}
    \centering
    \subfigure{\includegraphics[width=.28\textwidth]{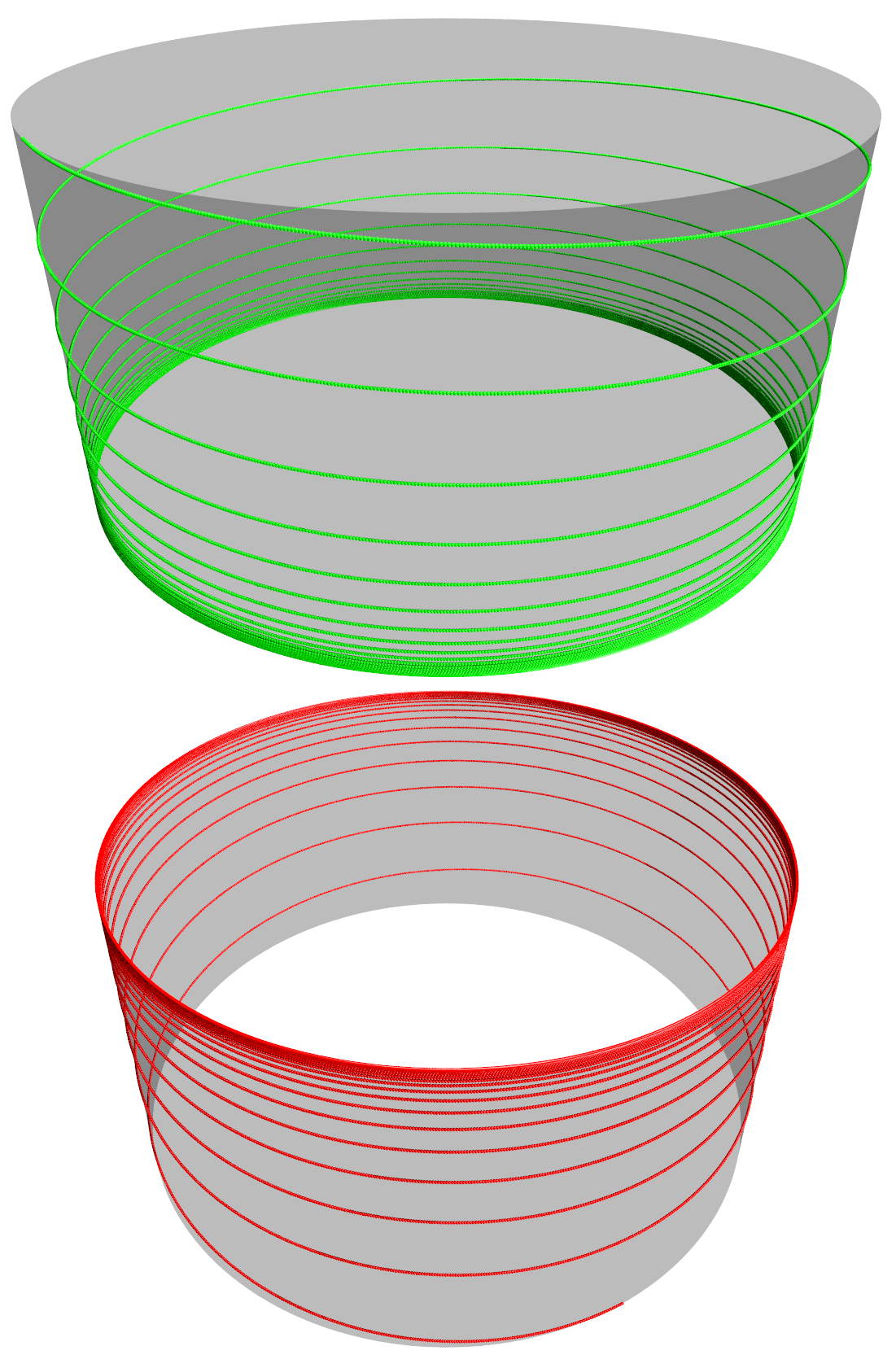} }
    \subfigure{\includegraphics[width=.68\textwidth]{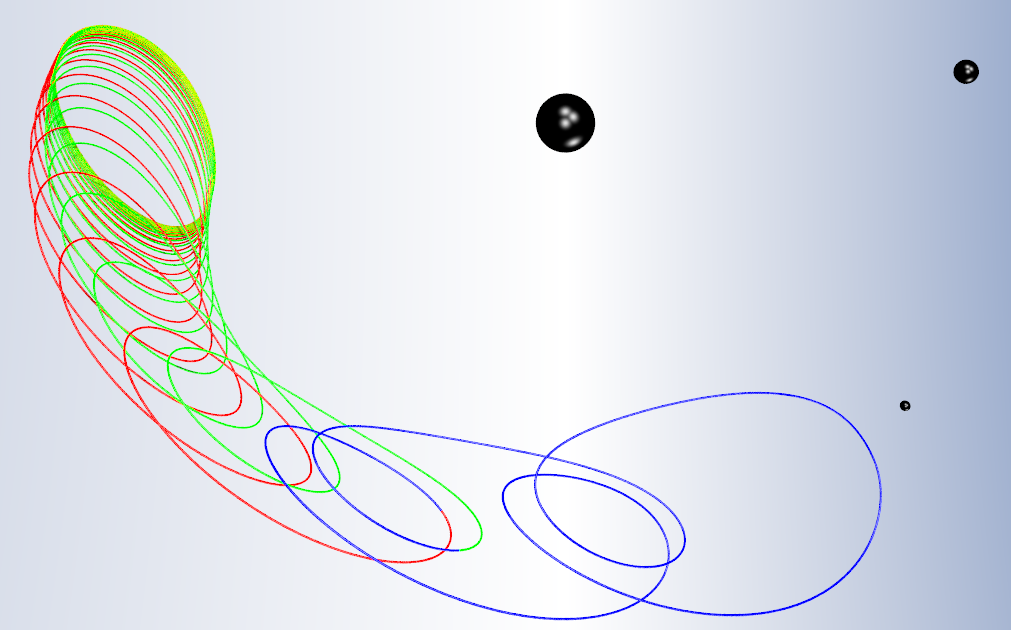} }
    \caption{Extension of the connecting orbit on the right in figure \ref{fig:ConnectingOrbit4B} using the parameterization of the manifolds and the conjugacy relation. 
    The red and green portions of the homoclinic are integration free. That is, they are obtained using the linear 
    dynamics in parameter space and exploiting the conjugacy given by the parameterization method.  
    On the left, the trajectories are displayed parameter space -- with the stable parameter cylinder above the unstable.
    Compare again with the Schematic of Figure $3$.}\label{fig:ExtendedConnecting4B}
\end{figure}

\section{Conclusion}
The methods of the present work facilitate accurate computation of local stable/unstable manifolds
attached to periodic orbits of vector fields in a large neighborhood of the orbit itself. 
The computations exploit Chebyshev expansions, so that domain decomposition 
can be used to improve the accuracy of the parameterization (compared to a Fourier-Taylor expansion) 
without necessarily increasing the total number spectral modes used. 
The results approximate the local manifolds in relatively large regions of phase space. 
The method is based on solving an invariance equation, so that the computations are
equipped with a convenient notion of defect/a-posteriori error.

The method is non-perturbative.  
So even though manifolds for Lyapunov orbits  
were computed in the CRTB and CRFB problems, the calculations 
do not use the fact that we were near an equilibrium. 
Since the parameterization method
is based on finding a zero of an invariance equation, and since we
use a Newton scheme to compute the numerical solution, it would 
be natural to develop numerical continuation methods
for the manifold computations.  Continuation in frequency, energy, or other system 
parameter would be natural.  
In this case the manifolds, and even the connecting orbits, 
do not have to be recomputed from scratch as 
parameters are changed.  Rather, the old orbits/manifolds can be used as
the initial guess for the Newton method at the new parameters.  
Utilizing a predictor/corrector scheme would also be natural.

An interesting topic for future work would be to compare the techniques developed 
in the present work with other techniques for computing high order expansions of 
local invariant manifolds attached periodic orbits.
For example, invariant objects similar to those discussed in Section \ref{sec:CRTBP} 
have been computed by a number of authors using methods based on Lindstedt-Poincar\'{e}
series or using high order normal forms  
\cite{MR3026237, MR2425052, MR2114685, MR2329518, MR1630282}.
One of the main differences between the methods of the works just cited, and 
the methods of the present work, is that both Lindstedt-Poincar\'{e} and normal 
form methods develop expansions valid in a full neighborhood of the periodic orbit. The 
stable/unstable manifolds are then obtained as suitable zero sections.  That is: the number of 
variables used in the expansion is equal to the dimension of the phase space rather than the 
dimension of the underlying invariant object.

Computing an expansion of a full neighborhood of the periodic orbit
is important for many applications.  For example
when designing a ``fly-by'' mission one wants to find
trajectories which approach the periodic orbit along (but not on) the stable manifold, 
and then move away after a finite time along (but not on) the unstable manifold. 
On the other hand if one is primarily interested in heteroclinic and homoclinic connections then
computing on a full neighborhood is much more expensive then just parameterizing the 
manifold as in the present work.  For example the parameterizations in Section \ref{sec:CRTBP} 
were computed using $N= 10$ Taylor order.  Since the parameterizations were expanded 
only in the stable or unstable direction, this requires computing only ten unknown Taylor coefficient
(each of which is a periodic function).  If one uses instead Lindstedt-Poincar\'{e} or normal 
form methods, then it is necessary to expand in one angle variable and three polynomial 
variables, and a polynomial of order $N=10$ in three variables has 
$286$ unknown coefficients (again, each of these is a periodic function). 
This back of the envelope comparison illustrates the advantage of parameterizing only 
the desired manifold, and not the full neighborhood when the particular application allows.

\begin{figure}
    \includegraphics[width=.6\textwidth]{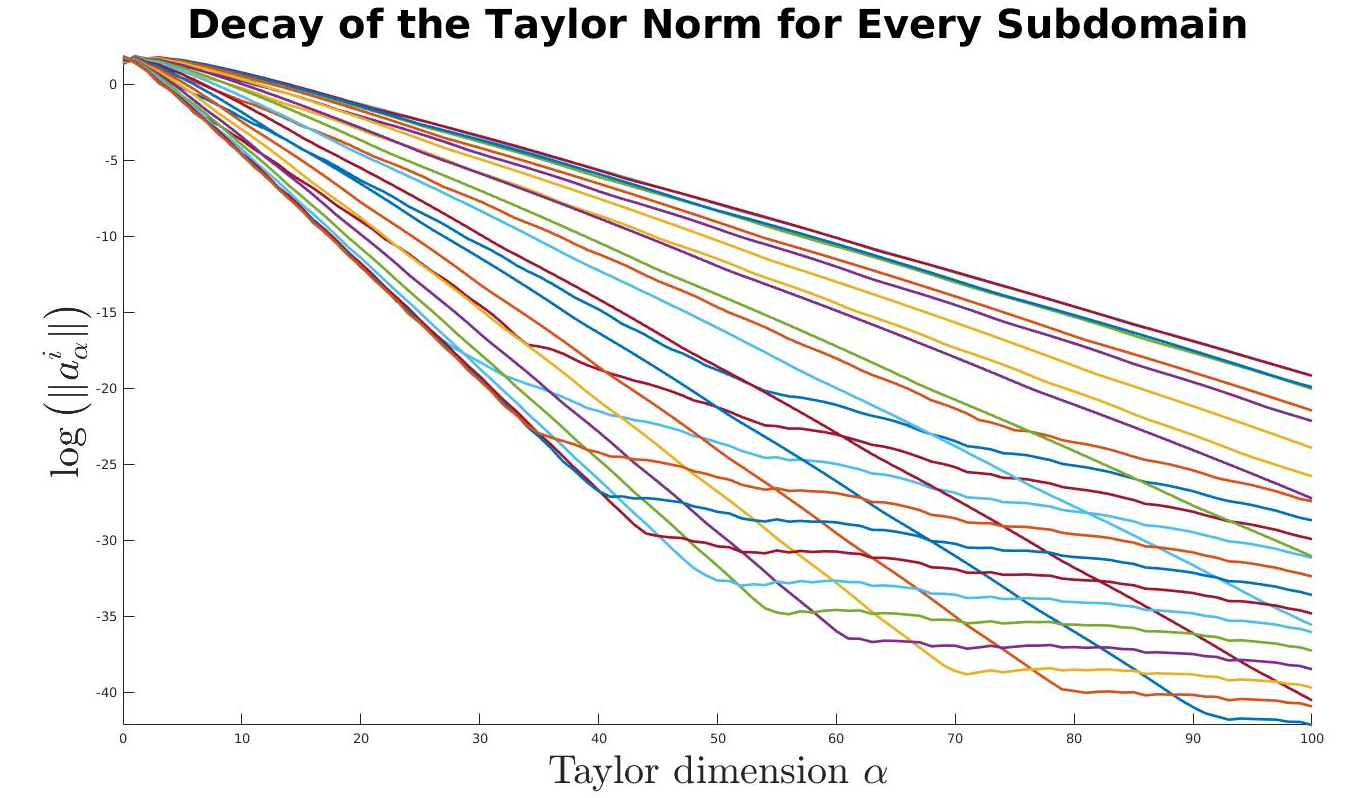}
    \caption{Decay of the logarithm of the norm of every subdomain. 
    The decay becomes slower once the norm reach machine precision. 
    The ``knee'' in the decay rates is due to the use of fast Fourier transform to compute convolution products, 
    which tends to stop or slow the decay after machine precision is reached.}\label{fig:DecayProfile}
\end{figure}

Another interesting possibility for future improvement is to study more
carefully the effects of a non-uniform subdivision strategy
for choosing the Chebyshev domains.  
In figure \ref{fig:DecayProfile}, we computed the norm of each sequence of Chebyshev expansion 
in the case of the stable manifold of the orbit AB displayed in figure \ref{fig:ABstableManifoldWithOrbits}
\[
\|a_\alpha^i\| = \sum_{j=1}^3 \sum_{k=0}^{m-1} \left| a_{\alpha,k}^{(i,j)} \right|,
\]
for all $\alpha=0,\hdots,100$ and $i=1,\hdots,50$. The scale of the eigenvector was chosen so that the norm of the last Taylor dimension is below machine precision in every subdomain. However, one can see from the figure that in this case several component were reaching this magnitude much earlier than some other ones. Such differences arise from the fact that the mesh was uniform in this case. One way to obtain bigger manifold without increasing the number of modes would be to use mesh adaptation, as in \cite{RayJB}. 
We also remark that there is a possibility that the computations could be sped up by pre-computing
the Floquet normal form as in 
\cite{MR3304254}.  The Floquet normal form would have to be discretized using 
Chebyshev rather than Fourier series, and we have not yet explored this possibility.

\nonumsection{Acknowledgments}
\noindent The authors wish to thank Jaime Burgos-Garc\'{i}a and J.P. Lessard for helpful conversations,
and to thank an anonymous referee for carefully reading the original submission and making 
a number of helpful comments and suggestions.  The final version of the manuscript is improved thanks to 
these efforts.  Both Maxime Murray and J.D. Mireles James were 
partially supported by NSF grant DMS-1700154, and by  the Alfred P. Sloan Foundation grant G-2016-7320.

\newpage

\bibliographystyle{apalike}
\bibliography{references}

\end{document}